\documentclass[12pt,reqno]{amsart}
\usepackage{amsmath} 

\usepackage{amsthm}
\usepackage{amssymb}
\usepackage{graphics}
\usepackage{latexsym}
\usepackage{comment}
\usepackage{lineno}
\usepackage{color}
\numberwithin{equation}{section}
\newtheorem{thm}{Theorem}[section]
\newtheorem{prop}[thm]{Proposition}
\newtheorem{lemm}[thm]{Lemma}
\newtheorem{cor}[thm]{Corollary}

\theoremstyle{remark}
\newtheorem{rem}{Remark}[section]
\newtheorem{defn}{Definition}

\newcommand{\BBB}{\mathbb}
\newcommand{\R}{{\BBB R}}
\newcommand{\Z}{{\BBB Z}}

\newcommand{\N}{{\BBB N}}
\newcommand{\C}{{\BBB C}}
\newcommand{\s}{\mathcal{S}}

\newcommand{\ee}{\mbox{\boldmath $1$}}

\newcommand{\LR}[1]{{\langle {#1} \rangle }}

\newcommand{\cross}{\times}

\newcommand{\al}{\alpha}
\newcommand{\be}{\beta}
\newcommand{\ga}{\gamma}

\newcommand{\e}{\varepsilon}
\newcommand{\ta}{\tau}

\newcommand{\p}{\partial}

\newcommand{\supp}{\operatorname{supp}}

\newcommand{\kuuhaku}{\text{}}

\newcommand{\F}{\mathcal{F}}
\newcommand{\1}{{\mathbf 1}}

\title[Well-posedness for system of qDNLS]{Well-posedness for a 
system of quadratic derivative
nonlinear Schr\"odinger equations with radial initial data
}

\author[H. Hirayama]{Hiroyuki Hirayama}
\address[H. Hirayama]{Organization for Promotion of Tenure Track, University of Miyazaki, 1-1, Gakuenkibanadai-nishi, Miyazaki, 889-2192 Japan}
\email[H. Hirayama]{h.hirayama@cc.miyazaki-u.ac.jp}

\author[S. Kinoshita]{Shinya Kinoshita}
\address[S. Kinoshita]{Universit\"at Bielefeld, 
Fakult\"at f\"ur Mathematik, Postfach 10 01 31 33501, Bielefeld, Germany
}
\email[S. Kinoshita]{kinoshita@math.uni-bielefeld.de}

\author[M. Okamoto]{Mamoru Okamoto}
\address[M. Okamoto]{Division of Mathematics and Physics, Faculty of Engineering, 
Shinshu University, 4-17-1, Wakasato, Nagano, 380-8553, Japan}
\email[M. Okamoto]{m\_okamoto@shinshu-u.ac.jp}


\begin{document}
%\linenumbers
%
%
\begin{abstract}
In the present paper, 
we consider the Cauchy problem of the system of quadratic derivative nonlinear 
Schr\"odinger equations. 
This system was introduced by M. Colin and T. Colin (2004). 
The first and second authors obtained some well-posedness 
results in the Sobolev space $H^{s}(\R^d)$. 
We improve these results for conditional radial initial data 
by rewriting the system radial form.  
\end{abstract}
\maketitle
%\tableofcontents
\setcounter{page}{001}

%%%%%%%%%%%%%%%%%%%%%%%%%%%%%%%%%%%%%%%%%%%%%%%%%%%%%%%%%%%%%%%%%%%%%%%%%%%%%%%%%
%%%%%%%%%%%%%%%%%%%%%%%%%%%%%%%%%%%%%%%%%%%%%%%%%%%%%%%%%%%%%%%%%%%%%%%%%%%%%%%%%
%%%%%%%%%%%%%%%%%%%%%%%%%%%%%%%%%%%%  Section 1   %%%%%%%%%%%%%%%%%%%%%%%%%%%%%%%%%%
%%%%%%%%%%%%%%%%%%%%%%%%%%%%%%%%%%%%%%%%%%%%%%%%%%%%%%%%%%%%%%%%%%%%%%%%%%%%%%%%%
%%%%%%%%%%%%%%%%%%%%%%%%%%%%%%%%%%%%%%%%%%%%%%%%%%%%%%%%%%%%%%%%%%%%%%%%%%%%%%%%%

\section{Introduction\label{intro}}
We consider the Cauchy problem of the system of nonlinear Schr\"odinger equations:
\begin{equation}\label{NLS_sys}
\begin{cases}
\displaystyle (i\p_{t}+\alpha \Delta )u=-(\nabla \cdot w )v,
\ \ t>0,\ x\in \R^d,\\
\displaystyle (i\p_{t}+\beta \Delta )v=-(\nabla \cdot \overline{w})u,
\ \ t>0,\ x\in \R^d,\\
\displaystyle (i\p_{t}+\gamma \Delta )w =\nabla (u\cdot \overline{v}),
\ \ t>0,\ x\in \R^d,\\
(u, v, w)|_{t=0}=(u_{0}, v_{0}, w_{0})\in (H^s(\R^d))^d\times (H^s(\R^d))^d\times (H^s(\R^d))^d,
\end{cases}
\end{equation}
where $\al$, $\be$, $\ga\in \R\backslash \{0\}$  
and the unknown functions $u$, $v$, $w$ are $d$-dimensional complex vector valued. 
The system (\ref{NLS_sys}) was introduced by Colin and Colin in \cite{CC04} 
as a model of laser-plasma interaction. (See, also \cite{CC06}, \cite{CCO09_1}.)
They also showed that the 
local existence of the solution of (\ref{NLS_sys}) in $H^s(\R^d)$ for $s>\frac{d}{2}+3$. 
The system (\ref{NLS_sys}) is invariant under the following scaling transformation:
\begin{equation}\label{scaling_tr}
A_{\lambda}(t,x)=\lambda^{-1}A(\lambda^{-2}t,\lambda^{-1}x)\ \ (A=(u,v,w) ), 
\end{equation}
and the scaling critical regularity is $s_{c}=\frac{d}{2}-1$. 
We put
\begin{equation}\label{coeff}
\theta :=\alpha\beta\gamma \left(\frac{1}{\alpha}-\frac{1}{\beta}-\frac{1}{\gamma}\right),\ \ 
\kappa :=(\alpha -\beta)(\alpha -\gamma)(\beta +\gamma). 
\end{equation}
We note that $\kappa =0$ does not occur when $\theta \ge 0$ 
for $\al$, $\be$, $\ga\in \R\backslash \{0\}$. 

First, we introduce some known results for related problems. 
%
%\begin{comment}
The system (\ref{NLS_sys}) has quadratic nonlinear terms which contain a derivative. 
A derivative loss arising from the nonlinearity makes the problem difficult. 
In fact, Mizohata (\cite{Mi85}) considered the Schr\"odinger equation
\[
\begin{cases}
i\partial_{t}u-\Delta u=(b_{1}(x)\cdot \nabla ) u,\ t\in \R ,\ x\in \R^{d},\\
u(0,x)=u_{0}(x),\ x\in \R^{d}
\end{cases}
\]
and proved that the uniform bound
\[
\sup_{x\in \R^{n},\omega \in S^{n-1},R>0}\left| {\rm Re}\int_{0}^{R}b_{1}(x+r\omega )\cdot \omega dr\right| <\infty
\]
is a necessary condition for the $L^{2} (\R^d)$ well-posedness. 
Furthermore, Christ (\cite{Ch}) proved that the flow map of 
the nonlinear Schr\"odinger equation
\begin{equation}\label{1dqdnls}
\begin{cases}
i\partial_{t}u-\partial_{x}^{2}u=u\partial_{x}u,\ t\in \R,\ x\in \R,\\
u(0,x)=u_{0}(x),\ x\in \R
\end{cases}
\end{equation}
is not continuous on $H^{s} (\R^d)$ for any $s\in \R$. 
From these results, 
it is difficult to obtain the well-posedness 
for quadratic derivative nonlinear Schr\"odinger equation in general. 
While for the system of quadratic derivative nonlinear equations,
it is known that the well-posedness holds. 
%\end{comment}
%
In \cite{Hi}, the first author proved the well-posedness of (\ref{NLS_sys}) 
in $H^s(\R^d)$, where $s$ is given in Table~\ref{WP_NLS_sys} below. 
\begin{table}[h]
\begin{center}
\begin{tabular}{|l|l|l|l|l|}
\hline
\multicolumn{2}{|c|}{} & $d=1$ & $d=2,3$ & $d\geq 4$\\
\hline
\multicolumn{2}{|c|}{$\theta > 0$} & \multicolumn{1}{|c|}{WP for $s\geq 0$} &\multicolumn{1}{|c|}{WP for $s\geq s_{c}$} & \multicolumn{1}{|c|}{WP for $s\geq s_{c}$}\\              
\cline{1-4}
\multicolumn{2}{|c|}{$\theta =0$} & \multicolumn{1}{|c|}{WP for $s\geq 1$} & 
WP for $s\geq 1$ & \multicolumn{1}{|c|}{}\\
\cline{2-3}
\multicolumn{2}{|c|}{$\kappa \ne 0$\ and\ $\theta < 0$} & \multicolumn{1}{|c|}{WP for $s\geq \frac{1}{2}$} & &\multicolumn{1}{|c|}{}\\
\hline
\end{tabular}
\caption{Well-posedness (WP for short)  for (\ref{NLS_sys}) proved in \cite{Hi}}
\end{center}
\end{table}\label{WP_NLS_sys}

Recently in \cite{HK}, 
the first and  second authors  improved this result 
by using the generalization of the Loomis-Whitney inequality 
introduced in \cite{BHT10} and \cite{BKW}.  
They proved the
well-posedness of (\ref{NLS_sys}) in $H^s(\R^d)$
for $s\ge \frac{1}{2}$ if $d=2$ and $s>\frac{1}{2}$ if $d=3$, 
under the condition $\kappa \ne 0$ and $\theta < 0$. 
While in \cite{Hi}, the first author also proved that 
the flow map is not $C^2$ for $s<1$ if $\theta = 0$ 
and for $s<\frac{1}{2}$ if $\theta < 0$ and $\kappa \ne 0$. 
Therefore, the well-posedness obtained in \cite{Hi} and \cite{HK} 
are optimal except the case $d=3$ and $s=\frac{1}{2}$ (which is scaling critical)
as far as we use the iteration argument. 
In particular, the optimal regularity are far from the scaling critical regularity 
if $d\le 3$ and $\theta \le 0$. 

%\begin{comment}
We point out that the results in \cite{Hi} and \cite{HK} do not 
contain the scattering of the solution for $d\le 3$ under the condition $\theta =0$ 
(and also $\theta <0$). 
In \cite{IKS13}, Ikeda, Katayama, and Sunagawa considered 
the system of quadratic nonlinear Schr\"odinger equations
\begin{equation}\label{qDNLS}
\left(i\partial_t+\frac{1}{2m_j}\Delta\right)u_j=F_j(u,\partial_xu),\ \ t>0,\ x\in \R^d,\ j=1,2,3, 
\end{equation}
under the mass resonance condition 
$m_1+m_2=m_3$ (which corresponds to the condition $\theta =0$ for (\ref{NLS_sys})), 
where $u=(u_1,u_2,u_3)$ is $\C^3$-valued, 
$m_1$, $m_2$, $m_3\in \R\backslash \{0\}$, and $F_j$ is defined by
\begin{equation}\label{qDNLS_nonlin}
\begin{cases}
F_{1}(u,\partial_xu)=\sum_{|\alpha |, |\beta|\le 1}
C_{1,\alpha,\beta}(\overline{\partial^{\alpha}u_2})(\partial^{\beta}u_3),\\
F_{2}(u,\partial_xu)=\sum_{|\alpha |, |\beta|\le 1}
C_{1,\alpha,\beta}(\partial^{\beta}u_3)(\overline{\partial^{\alpha}u_1}),\\
F_{3}(u,\partial_xu)=\sum_{|\alpha |, |\beta|\le 1}
C_{1,\alpha,\beta}(\partial^{\alpha}u_1)(\partial^{\beta}u_2)
\end{cases}
\end{equation}
with some constants $C_{1,\alpha,\beta}$, $C_{2,\alpha,\beta}$, $C_{3,\alpha,\beta}\in \C$. 
They obtained the small data global existence and the scattering 
of the solution to (\ref{qDNLS})
in the weighted Sobolev space for $d=2$ 
under the mass resonance condition
and the null condition for the nonlinear terms (\ref{qDNLS_nonlin}). 
They also proved the same result for $d\ge 3$ without the null condition. 
In \cite{IKO16}, Ikeda, Kishimoto, and Okamoto proved
the small data global well-posedness and the scattering of the solution 
to (\ref{qDNLS}) in $H^s (\R^d)$ for $d\ge 3$ and $s\ge s_c$ 
under the mass resonance condition
and the null condition for the nonlinear terms (\ref{qDNLS_nonlin}). 
They also proved the local well-posedness in $H^s (\R^d)$
for $d=1$ and $s\ge 0$, $d=2$ and $s>s_c$, and $d=3$ and $s\ge s_c$ 
under the same conditions. 
(The results in \cite{Hi} for $d\le 3$ and $\theta =0$ 
say that if the nonlinear terms do not have null condition, 
then $s=1$ is optimal regularity to obtain the well-posedness 
by using the iteration argument. ) 

Recently in \cite{SaSu}, 
Sakoda and Sunagawa considered 
(\ref{qDNLS}) for $d=2$ and $j=1,\cdots, N$ with
\begin{equation}\label{qDNLS_nonlin_2}
F_j(u,\partial_xu)
=\sum_{|\alpha|, |\beta|\le 1}\sum_{1\le k,l\le 2N}
C^{\alpha, \beta}_{j,k,l}(\partial_x^{\alpha}u_k^{\#})(\partial_x^{\beta}u_l^{\#}),
\end{equation}
where $u_j^{\#}=u_j$ if $j=1,\cdots ,N$, and 
$u_j^{\#}=\overline{u_j}$ if $j=N+1,\cdots ,2N$. 
They obtained the small data global existence and the time decay estimate 
for the solution
under some conditions for $m_1,\cdots m_N$ 
and the nonlinear terms (\ref{qDNLS_nonlin_2}), 
where the conditions contain (\ref{NLS_sys}) with $\theta =0$. 
While, it is known that the existence of the blow up solutions 
for the system of nonlinear Schr\"odinger equations. 
Ozawa and Sunagawa (\cite{OS13}) gave the examples of the derivative nonlinearity which causes the small data blow up for a system of Schr\"odinger equations. 
There are also some known results for a system of nonlinear Schr\"odinger equations with no derivative nonlinearity 
(\cite{HLN11}, \cite{HLO11}, \cite{HOT13}). 
%\end{comment}
%
%
%
%
%

The aim in the present paper is 
to improve the results in \cite{Hi} and \cite{HK} for conditional radial initial data in $\R^2$ and $\R^3$. 
The radial solution to (\ref{NLS_sys}) is only trivial solution 
since the nonlinear terms of (\ref{NLS_sys}) are not radial form. 
Therefore, we rewrite (\ref{NLS_sys}) into a radial form. 
%We also define $\dot{H}^s_{\rm rad}(\R^2)$ by the same manner 
Here, we focus on $d=2$. 
Let $\s (\R^2)$ denote the Schwartz class. 
If $w=(w_1,w_2)\in (\s (\R^2))^2$ satisfies
\begin{equation}\label{rot_cond}
\xi^{\perp}\cdot \widehat{w}(\xi)=\xi_1\widehat{w_2}(\xi)-\xi_2\widehat{w_1}(\xi)=0,\ \ 
x^{\perp}\cdot w(x)=x_1w_2(x)-x_2w_1(x)=0
\end{equation}
for any $\xi=(\xi_1,\xi_2)\in \R^2$ and $x=(x_1,x_2)\in \R^2$, 
then there exists a
scalar potential $W\in C^1(\R^2)$ 
satisfying 
\begin{equation}\label{sc_po}
\nabla W (x)=w(x),\ \ {}^{\forall}x\in \R^2
\end{equation}
and
\begin{equation}\label{sc_po_rad}
\frac{\partial}{\partial \vartheta}W(r\cos \vartheta, r\sin \vartheta )=0,\ \ 
{}^{\forall}(r,\vartheta)\in [0,\infty)\times [0,2\pi).
\end{equation}
Indeed, if we put
\[
W(x):=\int_{a_1}^{x_1}w_1(y_1,x_2)dy_1+\int_{a_2}^{x_2}w_2(a_1,y_2)dy_2
\]
for some $a_1$, $a_2\in \R$, then $W$ satisfies (\ref{sc_po}) 
by the first equality in (\ref{rot_cond}). 
Furthermore, $W$ also satisfies (\ref{sc_po_rad})
by the second equality in (\ref{rot_cond}). 
We note that the first equality in (\ref{rot_cond}) is equivalent to
\[
\nabla^{\perp}\cdot w(x)=\partial_1w_2(x)-\partial_2w_1(x)=0, 
\]
which is the irrotational condition. 
\begin{comment}
Furthermore, we have
\[
\|W\|_{\dot{H}^{s+1}}+\|W\|_{\dot{H}^1}
\sim \|\nabla W\|_{\dot{H}^{s}}+\|\nabla W\|_{L^2}
\sim \|w\|_{H^s}<\infty
\]
for $s\ge 0$. 
\begin{rem}
Since $w_1$, $w_2\in L^2(\R^2)$, it holds that
\[
\int_{-\infty}^{\infty}|w_1(y_1,x_2)|^2dy_1<\infty\ {\rm a.e.}\ x_2\in \R,\ \ 
\int_{-\infty}^{\infty}|w_2(x_1,y_2)|^2dy_2<\infty\ {\rm a.e.}\ x_1\in \R. 
\]
Therefore, we have
\[
\begin{split}
\int_{a_1}^{x_1}|w_1(y_1,x_2)|dx_2&\le |x_1-a_1|^{\frac{1}{2}}
\left(\int_{-\infty}^{\infty}|w_1(y_1,x_2)|^2dy_1\right)^{\frac{1}{2}}
<\infty\ {\rm a.e.}\ x\in \R^2,\\
\int_{a_2}^{x_2}|w_2(x_1,y_2)|dy_2&\le |x_2-a_2|^{\frac{1}{2}}
\left(\int_{-\infty}^{\infty}|w_2(x_1,y_2)|^2dy_2\right)^{\frac{1}{2}}
<\infty\ {\rm a.e.}\ x\in \R^2
\end{split}
\]
by the Cauchy-Schwarz inequality. 
This means that $W(x)$ is well-defined for almost all $x\in \R^2$.
\end{rem}
\end{comment}
%
\begin{rem}
If $d=3$, we can also obtain the radial scalar potential 
$W\in C^1(\R^3)$ 
of $w=(w_1,w_2,w_3)\in (\s (\R^3))^3$
by assuming the conditions
\begin{equation}\label{rot_cond_3d}
\xi \times \widehat{w}(\xi)=0,\ x\times w(x)=0
\end{equation}
instead of (\ref{rot_cond}). 
\end{rem}
\begin{defn}\label{radial_def}
We say $f\in \s'(\R^d)$ is radial if it holds that
\[
<f,\varphi \circ R>=<f,\varphi>
\]
for any $\varphi \in \s (\R^d)$ and rotation 
$R:\R^d\rightarrow \R^d$. 
\end{defn}
\begin{rem}
If $f\in L^1_{\rm loc}(\R^d)$, then Definition~\ref{radial_def} 
is equivalent to
\[
{}^\exists g:\R\rightarrow \C\ {\rm s.t.}\  f(x)=g(|x|),\ \ {\rm a.e.}\ x\in \R^d.
\]
\end{rem}
Now, we consider the system of 
nonlinear Schr\"odingers equations:
\begin{equation}\label{NLS_rad}
\begin{cases}
\displaystyle (i\p_{t}+\alpha \Delta )u=-(\Delta W )v,
\ \ t>0,\ x\in \R^d,\\
\displaystyle (i\p_{t}+\beta \Delta )v=-(\Delta \overline{W})u,
\ \ t>0,\ x\in \R^d,\\
\displaystyle (i\p_{t}+\gamma \Delta )\nabla W =\nabla (u\cdot \overline{v}),
\ \ t>0,\ x\in \R^d,\\
(u, v, [W])|_{t=0}=(u_{0}, v_{0}, [W_{0}])\in \mathcal{H}^s(\R^d)
\end{cases}
\end{equation}
instead of (\ref{NLS_sys}), where $d=2$ or $3$, and 
\[
\begin{split}
\mathcal{H}^s(\R^d)&:=(H^s_{\rm rad}(\R^d))^d\times (H^s_{\rm rad}(\R^d))^d\times 
\widetilde{H}^{s+1}_{\rm rad}(\R^d),\\
H^s_{\rm rad}(\R^d)
&:=\{f\in H^s(\R^d)|\ f\ {\rm is\ radial}\},\\
\widetilde{H}^{s+1}(\R^d)&:=\{f\in \s'(\R^d)| \ \nabla f\in (H^s(\R^d))^d\}/\mathcal{N}_0,\\
\mathcal{N}_0
&:= \{f\in \s'(\R^d)|\ \nabla f=0 \}, \\
\widetilde{H}^{s+1}_{\rm rad}(\R^d)
&:=\{[f]\in \widetilde{H}^{s+1}(\R^d)|\ f\ {\rm is\ radial}\}.
\end{split}
\]
The norm for an equivalent class $[f]\in \widetilde{H}^{s+1}(\R^d)$ is 
defined by 
\[
\|[f]\|_{\widetilde{H}^{s+1}}:=\|\nabla f\|_{(H^s)^d}\sim \|f\|_{\dot{H}^{s+1}}
+\|f\|_{\dot{H}^1},
\] 
which is well-defined since $ \widetilde{H}^{s+1}(\R^d)$ 
is a quotient space.
%For simplicity, we identify $[f]$ and $f$, and write $f$ instead of $[f]$. 
The system (\ref{NLS_rad}) is obtained by substituting 
$w=\nabla W$ and $w_0=\nabla W_0$ in (\ref{NLS_sys}). 
\begin{defn}
We say $(u,v,[W])\in C([0,T];\mathcal{H}^s(\R^d))$ is a solution 
to (\ref{NLS_rad}) if 
\[
\begin{split}
u(t)&=e^{it\alpha \Delta}u_0+i\int_0^te^{i(t-t')\alpha \Delta}(\Delta W(t'))v(t')dt'\ \ 
{\rm in}\ (H^s(\R^d))^d,\\
v(t)&=e^{it\beta \Delta}v_0+i\int_0^te^{i(t-t')\beta \Delta}(\Delta \overline{W(t')})v(t')dt'\ \ {\rm in}\ (H^s(\R^d))^d,\\
\nabla W(t)&=e^{it\gamma \Delta}\nabla W_0-i\int_0^te^{i(t-t')\gamma \Delta}\nabla (u(t')\cdot \overline{v(t')})dt'\ \ {\rm in}\ H^s(\R^d)
\end{split}
\]
hold for any $t\in [0,T]$. 
This definition does not depend on 
how we choose a representative $W$. 
\end{defn}

Now, we give the main results in this paper. 
\begin{thm}\label{wellposed_1}
Assume $\kappa \ne 0$. \\
{\rm (i)}\ Let $d=2$. Assume that 
$s\ge \frac{1}{2}$ if $\theta =0$ and $s>0$ if $\theta <0$. 
Then, {\rm (\ref{NLS_rad})} is locally well-posed in 
$\mathcal{H}^s(\R^2)$. \\
{\rm (ii)}\ Let $d=3$. 
Assume that $\theta \le 0$ and $s\ge \frac{1}{2}$. 
Then, {\rm (\ref{NLS_rad})} is locally well-posed in 
$\mathcal{H}^s(\R^3)$. \\
{\rm (iii)}\ Let $d=3$. 
Assume that $\theta \le 0$ and $s\ge \frac{1}{2}$. 
Then, {\rm (\ref{NLS_rad})} is globally well-posed in 
$\mathcal{H}^s(\R^3)$ for small data. 
Furthermore, the solution scatters in $\mathcal{H}^s(\R^3)$. 
\end{thm}
\begin{rem}
$s=0$ for $d=2$, and $s=\frac{1}{2}$ for $d=3$ are 
scaling critical regularity for (\ref{NLS_sys}). 
\end{rem}
While, we obtain the following. 
\begin{thm}\label{ill-posed}
Let $d=2$  
and $\theta =0$. 
Then, the flow map of (\ref{NLS_rad}) is not $C^2$ in $\mathcal{H}^s(\R^2)$ for $s<\frac{1}{2}$. 
\end{thm}
\begin{rem}
Theorem~\ref{ill-posed} says that the well-posedness in Theorem~\ref{wellposed_1} 
for $\theta =0$ is optimal as far as we use the iteration argument. 
%Since $d=2$ and $s=0$ is scaling critical, 
%the well-posedness for $\theta <0$ 
%is also optimal except the critical case. 
\end{rem}
\begin{rem}
It is interesting that the result for 2D radial initial data is better than 
that for 1D initial data. 
Actually, the optimal regularity for 1D initial data is 
$s= 1$ if $\theta =0$, and $s= \frac{1}{2}$ if $\theta <0$ and $\kappa \ne 0$, 
which are larger than the optimal regularity for 2D radial initial data. 
The reason is the following. 
We use the angular decomposition and each angular localized 
term has a better property. 
For radial functions, the angular localized bound leads to 
an estimate for the original functions. (See, (\ref{rad_l2est}) below.)
\end{rem}
We note that if $\nabla W_0=w_0$ holds and
$(u,v,[W])$ is a solution to (\ref{NLS_rad}) 
with $(u,v,[W])|_{t=0}=(u_0,v_0,[W_0])\in \mathcal{H}^s(\R^d)$, then 
$(u,v,\nabla W)$ is a solution to (\ref{NLS_sys}) 
with $(u,v,\nabla W)|_{t=0}=(u_0,v_0,w_0)\in  (H^s_{\rm rad}(\R^d))^d\times (H^s_{\rm rad}(\R^d))^d\times H^s(\R^d)$. 
The existence of a scalar potential $W_0\in \widetilde{H}^{s+1}_{\rm rad}(\R^d)$ 
will be proved for $w_0\in  \mathcal{A}^s(\R^d)$ with $s>\frac{1}{2}$ (See, Proposition~\ref{ex_sc}), where
\[
\begin{split}
\mathcal{A}^s(\R^2)
&:=\{f=(f_1,f_2)\in (H^s(\R^2))^2|\ f\ {\rm satisfies}\ (\ref{rot_cond})\ {\rm a.e.}\ 
x, \xi \in \R^2\},\\
\mathcal{A}^s(\R^3)
&:=\{f=(f_1,f_2,f_3)\in (H^s(\R^3))^3|\ f\ {\rm satisfies}\ (\ref{rot_cond_3d})\ {\rm a.e.}\ 
x, \xi \in \R^3\}.
\end{split}
\]
Therefore, we obtain the following.
\begin{thm}\label{wellposed_2}
Let $d=2$ or $3$. 
Assume that $\theta = 0$ and $s>\frac{1}{2}$. 
Then, {\rm (\ref{NLS_sys})} is locally well-posed in 
$(H^s_{\rm rad}(\R^d))^d\times (H^s_{\rm rad}(\R^d))^d\times \mathcal{A}^s(\R^d)$. 
\end{thm}
\begin{rem}
For $d=3$, Theorem~\ref{wellposed_1} can be obtained by almost the same 
way as in \cite{Hi}. 
In Proposition~4.4 (i) of \cite{Hi} , the author used 
the Strichartz estimate
\[
\|e^{it\Delta}P_{N}u_0\|_{L^q_tL^r_x(\R\times \R^d)}
\lesssim \|P_{N}u_0\|_{L^2}
\]
and 
\[
\left|N_{\rm max}\int_0^T\int_{\R^d}(P_{N_1}u_1)(P_{N_2}u_2)(P_{N_3}u_3)dxdt\right|
\lesssim 
N_{\rm \max}^{s_c}\prod_{j=1}^3\|P_{N_j}u_j\|_{L^q_tL^r_x}
\]
with an admissible pair $(q,r)=(3,\frac{6d}{3d-4})$ for $d\ge 4$. 
But this trilinear estimate does not hold for $d=3$. 
This is the reason why the well-posedness in $H^{s_c}(\R^3)$ 
could not be obtained in \cite{Hi}. 
For the radial function $u_0\in L^2(\R^3)$, it is known that the improved 
Strichartz estimate (\cite{Sh09}, Corollary~6.2)
\[
\|e^{it\Delta}P_{N}u_0\|_{L^3_{t,x}(\R\times \R^3)}
\lesssim N^{-\frac{1}{6}}\|P_{N}u_0\|_{L^2}.
\]
While, it holds that
\[
\left|N_{\rm max}\int_0^T\int_{\R^3}(P_{N_1}u_1)(P_{N_2}u_2)(P_{N_3}u_3)dxdt\right|
%\lesssim N_{\rm \max}\prod_{j=1}^3\|P_{N_j}u_j\|_{L^3_{t,x}}
\lesssim N_{\rm max}^{\frac{1}{2}}\prod_{j=1}^3N_j^{\frac{1}{6}}\|P_{N_j}u_j\|_{L^3_{t,x}}
\]
for $N_1\sim N_2\sim N_3\ge 1$. 
Therefore, for $d=3$, we can obtain the same estimate 
in Proposition~4.4 (i). 
Because of  such reason, 
we omit more detail of the proof for $d=3$, 
and only consider $d=2$ in the following sections.  
\end{rem}
%
%
%To give the main results of the present paper, 
%
\begin{comment}
\begin{rem}
The initial datum and the solutions for (\ref{NLS_sys}) 
are $\C^d$-valued function. 
Therefore, $u\in X^{s,b,p}_{\sigma}$ means 
$u^{(j)}\in X^{s,b,p}_{\sigma}$\ ($j=1,\cdots, d)$ and 
$\|u\|_{X^{s,b,p}_{\sigma}}$ means $\sum_{j=1}^d\|u^{(j)}\|_{X^{s,b,p}_{\sigma}}$ 
for $u=(u^{(1)},\cdots,u^{(d)})$
in this paper. 
Similarly, $u_0\in H^s$ means $u_0^{(j)}\in H^s$\ ($j=1,\cdots, d)$ and 
$\|u_0\|_{H^s}$ means $\sum_{j=1}^d\|u_0^{(j)}\|_{H^s}$ 
for $u_0=(u_0^{(1)},\cdots,u_0^{(d)})$. 
\end{rem}
\end{comment}
%
%
%
%
%
\begin{comment}
We make a comment on Theorem \ref{wellposed_1}. 
In \cite{Hi}, the first author proved that 
the flow map is not $C^2$ for $s<1/2$ 
under the condition $\theta <0$ and $\kappa \ne 0$. 
Therefore, the above result is optimal as long as we use the iteration argument. 
In \cite{Hi}, we needed the condition $s \geq 1$ to show the key nonlinear estimates for ``resonance'' interactions which is the most difficult interactions to estimate since we cannot recover a derivative loss from modulations. 
To overcome this, we employ a new estimate which was introduced in \cite{BKW}, \cite{BHT10} 
and applied to the Zakharov system in \cite{BHHT09} and \cite{BH11}. 
See Proposition \ref{prop2.7} below.
%
%
\end{comment}
%
\noindent {\bf Notation.} 
We denote the spatial Fourier transform by\ \ $\widehat{\cdot}$\ \ or $\F_{x}$, 
the Fourier transform in time by $\F_{t}$ and the Fourier transform in all variables by\ \ $\widetilde{\cdot}$\ \ or $\F_{tx}$. 
For $\sigma \in \R$, the free evolution $e^{it\sigma \Delta}$ on $L^{2}$ is given as a Fourier multiplier
\[
\F_{x}[e^{it\sigma \Delta}f](\xi )=e^{-it\sigma |\xi |^{2}}\widehat{f}(\xi ). 
\]
We will use $A\lesssim B$ to denote an estimate of the form $A \le CB$ for some constant $C$ and write $A \sim B$ to mean $A \lesssim B$ and $B \lesssim A$. 
We will use the convention that capital letters denote dyadic numbers, e.g. $N=2^{n}$ for $n\in \N_0:=\N\cup\{0\}$ and for a dyadic summation we write
$\sum_{N}a_{N}:=\sum_{n\in \N_0}a_{2^{n}}$ and $\sum_{N\geq M}a_{N}:=\sum_{n\in \N_0, 2^{n}\geq M}a_{2^{n}}$ for brevity. 
Let $\chi \in C^{\infty}_{0}((-2,2))$ be an even, non-negative function such that $\chi (t)=1$ for $|t|\leq 1$. 
We define $\psi (t):=\chi (t)-\chi (2t)$, 
$\psi_1(t):=\chi (t)$, and $\psi_{N}(t):=\psi (N^{-1}t)$ for $N\ge 2$. 
Then, $\sum_{N}\psi_{N}(t)=1$.  
We define frequency and modulation projections
\[
\widehat{P_{N}u}(\xi ):=\psi_{N}(\xi )\widehat{u}(\xi ),\ 
\widetilde{Q_{L}^{\sigma}u}(\tau ,\xi ):=\psi_{L}(\tau +\sigma |\xi|^{2})\widetilde{u}(\tau ,\xi ).
\]
Furthermore, we define $Q_{\geq M}^{\sigma}:=\sum_{L\geq M}Q_{L}^{\sigma}$ and $Q_{<M}:=Id -Q_{\geq M}$. 

The rest of this paper is planned as follows.
In Section 2, we will give the bilinear estimates which will be used to prove the well-posedness.
In Section 3, we will give the proof of Theorems~\ref{wellposed_1} and ~\ref{wellposed_2}. 
In Section 4, we will give the proof of Theorem~\ref{ill-posed}. 
\section{Bilinear estimates \label{be_for_2d}}
%
%%%%%%%%%%%%%%%%%%%%%%%%%%%%%%%%%%%%%%%%%%%%%%%%%%%%%%%%%%%%%%%%%%%%%%%%%%%%%%%%%
%%%%%%%%%%%%%%%%%%%%%%%%%%%%%%%%%%%%%%%%%%%%%%%%%%%%%%%%%%%%%%%%%%%%%%%%%%%%%%%%%
%%%%%%%%%%%%%%%%%%%%%%%%%%%%%%%%%%%%  Section 2  %%%%%%%%%%%%%%%%%%%%%%%%%%%%%%%%%%
%%%%%%%%%%%%%%%%%%%%%%%%%%%%%%%%%%%%%%%%%%%%%%%%%%%%%%%%%%%%%%%%%%%%%%%%%%%%%%%%%
%%%%%%%%%%%%%%%%%%%%%%%%%%%%%%%%%%%%%%%%%%%%%%%%%%%%%%%%%%%%%%%%%%%%%%%%%%%%%%%%%
%
In this section, we prove the bilinear estimates. 
First, we define the radial condition for time-space function.  
\begin{defn}
We say $u\in \s'(\R_t\times \R_x^2)$ is radial with respect to $x$ if it holds that
\[
<u,\varphi_R>=<u,\varphi>
\]
for any $\varphi \in \s (\R_t\times \R_x^2)$ and rotation 
$R:\R^2\rightarrow \R^2$, where $\varphi_R\in \s (\R_t\times \R_x^2)$ 
is defined by $\varphi_R(t,x)=\varphi (t,R(x))$. 
\end{defn}
Next, we define the Fourier restriction norm, 
which was introduced by Bourgain in \cite{Bo93}. 
\begin{defn}
Let $s\in \R$, $b\in \R$, $\sigma \in \R\backslash \{0\}$. \\
(i)\ We define $X^{s,b}_{\sigma}:=\{u\in \s'(\R_t\times \R_x^2)|\ \|u\|_{X^{s,b}_{\sigma}}<\infty\}$,  
where
\[
\begin{split}
\|u\|_{X^{s,b}_{\sigma}}&:=
\|\langle \xi \rangle^s\langle \tau +\sigma |\xi|^2\rangle^b\widetilde{u}(\tau,\xi)\|_{L^2_{\tau\xi}}
\sim \left(\sum_{N\ge 1}
\sum_{L\ge 1}N^{2s}L^{2b}\|Q_{L}^{\sigma}P_{N}u\|_{L^2}^2\right)^{\frac{1}{2}}. 
\end{split}
\]
(ii)\ We define $\widetilde{X}^{s+1,b}_{\sigma}:=\{u\in \s'(\R_t\times \R_x^2)|\ \nabla u\in X^{s,b}_{\sigma}\}/ \mathcal{N}$ with the norm
\[
\|[u]\|_{\widetilde{X}^{s+1,b}_{\sigma}}:=\|\nabla u\|_{X^{s,b}_{\sigma}},
\]
where $\mathcal{N} := \{u\in \s'(\R_t\times \R_x^2)|\ \nabla u =0 \}$.
\\
(iii)\ We define
\[
\begin{split}
X^{s,b}_{\sigma,{\rm rad}}&:=
\{u\in X^{s,b}_{\sigma}|\ u\ {\rm is\ radial\ with\ respect\ to}\ x\},\\
\widetilde{X}^{s,b}_{\sigma,{\rm rad}}&:=
\{[u]\in \widetilde{X}^{s+1,b}_{\sigma}|\ u\ {\rm is\ radial\ with\ respect\ to}\ x\}.
\end{split}
\]
%(iv)\ For a Banach space $X$ and $T>0$, 
%we define the time localized space $X_{T}$ as
%\[
%X_{T}:=\{u|_{[0,T]}|\ u\in X\}
%\]
%with the norm
%\[
%\|u\|_{X_{T}}:=\inf \{\|v\|_{X}|\ v\in X,\ v|_{[0,T]}=u|_{[0,T]}\}.
%\]
%For simplicity, we identify $[u]\in \s'(\R_t\times \R_x^d)/\s'(\R_t)$ and 
%$u\in \s'(\R_t\times \R_x^d)$, and write $u$ instead of $[u]$.
\end{defn}
We put
\[
\widetilde{\theta} :=\sigma_1\sigma_2\sigma_3 \left(\frac{1}{\sigma_1}+\frac{1}{\sigma_2}+\frac{1}{\sigma_3}\right),\ \ 
\widetilde{\kappa} :=(\sigma_1+\sigma_2)(\sigma_2 +\sigma_3)(\sigma_3 +\sigma_1). 
\]
We note that if $(\sigma_1,\sigma_2,\sigma_3)\in \{(\beta, \gamma, -\alpha), 
(-\gamma, \alpha, -\beta), (\alpha, -\beta, -\gamma)\}$, 
then it hold that $\widetilde{\theta}=\theta$ and $|\widetilde{\kappa}|=|\kappa|$. 

The following bilinear estimate plays a central role to show Theorem \ref{wellposed_1}.
\begin{prop}\label{key_be}
Let $\sigma_1$, $\sigma_2$, $\sigma_3\in \R\backslash \{0\}$ satisfy 
$\widetilde{\kappa}\ne 0$. 
Let $s\ge \frac{1}{2}$ if $\widetilde{\theta}=0$ 
and $s>0$ if $\widetilde{\theta}<0$. 
Then there exists $b'\in (0,\frac{1}{2})$ and $C>0$ such that
\begin{align}
\||\nabla |(uv)\|_{X^{s,-b'}_{-\sigma_3}}
&\le C\|u\|_{X^{s,b'}_{\sigma_1}}\|v\|_{X^{s,b'}_{\sigma_2}},
 \label{be_111}\\
\|(\Delta U)v\|_{X^{s,-b'}_{-\sigma_3}}
&\le C(\|\partial_1 U\|_{X^{s,b'}_{\sigma_1}}
+\|\partial_2 U\|_{X^{s,b'}_{\sigma_1}})\|v\|_{X^{s,b'}_{\sigma_2}}
\label{be_222}
\end{align}
hold for any $u\in X^{s,b'}_{\sigma_1,{\rm rad}}$, 
$v\in X^{s,b'}_{\sigma_2,{\rm rad}}$, and $[U]\in \widetilde{X}^{s+1,b'}_{\sigma_1,{\rm rad}}$.
\begin{comment}
\begin{align*}
 \|\p_j (uv)\|_{X^{s,-b'}_{-\sigma_3}}
+ \|(\p_j u)  v\|_{X^{s,-b'}_{-\sigma_3}}  
+ \|u (\p_j v ) \|_{X^{s,-b',}_{-\sigma_3}}
\le C(\|u\|_{X^{s,b',1}_{\sigma_1}} & \|v\|_{X^{s,b',1}_{\sigma_2}},\ j=1,2. 
\end{align*}
hold for any $u\in X^{s,b'}_{\sigma_1,{\rm rad}}$ and $v\in X^{s,b'}_{\sigma_2,{\rm rad}}.
\end{comment}
\end{prop}
\begin{rem}
Since $\|\partial_1(uv)\|_{X^{s,-b'}_{-\sigma_3}}
+\|\partial_2(uv)\|_{X^{s,-b'}_{-\sigma_3}}\sim 
\||\nabla |(uv)\|_{X^{s,-b'}_{-\sigma_3}}$, 
(\ref{be_111}) implies
\[
\|\partial_1(uv)\|_{X^{s,-b'}_{-\sigma_3}}
+\|\partial_2(uv)\|_{X^{s,-b'}_{-\sigma_3}}
\le C\|u\|_{X^{s,b'}_{\sigma_1}}\|v\|_{X^{s,b'}_{\sigma_2}}. 
\]
\end{rem}
To prove Proposition~\ref{key_be}, we first give the Strichartz estimate. 
\begin{prop}[Strichartz estimate (cf. \cite{GV85}, \cite{KT98})]\label{Stri_est}
Let $\sigma \in \R\backslash \{0\}$ and $(p,q)$ be an admissible pair of exponents for the 2D Schr\"odinger equation, i.e. $p>2$, 
$\frac{1}{p}+\frac{1}{q}=\frac{1}{2}$. Then, we have
\[
\|e^{it\sigma \Delta}\varphi \|_{L_{t}^{p}L_{x}^{q}(\R\times \R^2)}\lesssim \|\varphi \|_{L^{2}_{x}(\R^2)}.
\]
for any $\varphi \in L^{2}(\R^{2})$. 
\end{prop}
The Strichartz estimate implies the following. 
(See the proof of Lemma\ 2.3 in \cite{GTV97}.)
\begin{cor}\label{Bo_Stri}
Let $L\in 2^{\N_0}$, $\sigma \in \R\backslash \{0\}$, 
and $(p,q)$ be an admissible pair of exponents for the Schr\"odinger equation. 
Then, we have
\begin{equation}\label{Stri_est_2}
\|Q_{L}^{\sigma}u\|_{L_{t}^{p}L_{x}^{q}}\lesssim L^{\frac{1}{2}}\|Q_{L}^{\sigma}u\|_{L^{2}_{tx}}.
\end{equation}
for any $u \in L^{2}(\R\times \R^{2})$. 
\end{cor}
Next, we give the bilinear Strichartz estimate.
\begin{prop}\label{L2be}
We assume that $\sigma_{1}$, $\sigma_{2}\in \R \backslash \{0\}$ 
satisfy $\sigma_1+\sigma_2\ne 0$. 
For any dyadic numbers $N_1$, $N_2$, $N_3\in 2^{\N_0}$ 
and $L_1$, $L_2\in 2^{\N_0}$, we have
\begin{equation}\label{L2be_est}
\begin{split}
&\|P_{N_3}(Q_{L_1}^{\sigma_1}P_{N_1}u_{1}\cdot Q_{L_2}^{\sigma_2}P_{N_2}u_{2})\|_{L^{2}_{tx}(\R\times \R^2)}\\
&\lesssim \left(\frac{N_{\min}}{N_{\max}}\right)^{\frac{1}{2}}L_1^{\frac{1}{2}}L_2^{\frac{1}{2}}
\|Q_{L_1}^{\sigma_1}P_{N_1}u_{1}\|_{L^2_{tx}(\R\times \R^2)}\|Q_{L_2}^{\sigma_2}P_{N_2}u_{2}\|_{L^2_{tx}(\R\times \R^2)}, 
\end{split}
\end{equation}
where $N_{\min}=\displaystyle \min_{1\le i\le 3}N_i$, 
$N_{\max}=\displaystyle \max_{1\le i\le 3}N_i$.
\end{prop}
Proposition~\ref{L2be} can be obtained by the same way 
as Lemma\ 1 in \cite{CDKS01}. (See, also Lemma 3.1 in \cite{Hi}.)

\begin{cor}\label{L2be_2}
Let $b'\in (\frac{1}{4},\frac{1}{2})$, 
and 
$\sigma_{1}$, $\sigma_{2}\in \R \backslash \{0\}$ satisfy $\sigma_1+\sigma_2\ne 0$, 
We put $\delta =\frac{1}{2}-b'$. 
For any dyadic numbers $N_1$, $N_2$, $N_3\in 2^{\N_0}$ 
and $L_1$, $L_2\in 2^{\N_0}$, we have
\begin{equation}\label{L2be_est_2}
\begin{split}
&\|P_{N_3}(Q_{L_1}^{\sigma_1}P_{N_1}u_{1}\cdot Q_{L_2}^{\sigma_2}P_{N_2}u_{2})\|_{L^{2}_{tx}(\R\times \R^2)}\\
&\lesssim 
N_{\min}^{4 \delta}
\left(\frac{N_{\min}}{N_{\max}}\right)^{\frac{1}{2}- 2\delta}L_1^{b'}L_2^{b'}
\|Q_{L_1}^{\sigma_1}P_{N_1}u_{1}\|_{L^2_{tx}(\R\times \R^2)}\|Q_{L_2}^{\sigma_2}P_{N_2}u_{2}\|_{L^2_{tx}(\R\times \R^2)}.  
\end{split}
\end{equation}
\end{cor}
The proof is given in Corollary~2.5 in \cite{HK}.
\subsection{The estimates for low modulation}
In this subsection, we assume that $L_{\textnormal{max}} \ll N_{\max}^2$. 
%In this case, we cannot recover a derivative loss by using $L_{\textnormal{max}} \gtrsim N_{\max}^2$. 
%Therefore, the strategy for the case $L_{\textnormal{max}} \gtrsim N_{\max}^2$ is no longer available. 
%However, thanks to $\widetilde{\kappa} \not= 0$, the following relation holds. 
%
\begin{lemm}\label{modul_est_2}
We assume that $\sigma_1$, $\sigma_2$, $\sigma_3 \in \R \setminus \{0 \}$ satisfy 
$\widetilde{\kappa}\neq 0$ and $(\tau_{1},\xi_{1})$, $(\tau_{2}, \xi_{2})$, $(\tau_{3}, \xi_{3})\in \R\times \R^{2}$ satisfy $\tau_{1}+\tau_{2}+\tau_{3}=0$, $\xi_{1}+\xi_{2}+\xi_{3}=0$. 
If $\displaystyle \max_{1\leq j\leq 3}|\tau_{j}+\sigma_{j}|\xi_{j}|^{2}|
\ll \max_{1\leq j\leq 3}|\xi_{j}|^{2}$ then we have
\begin{equation*}
|\xi_1| \sim |\xi_2| \sim |\xi_3|.
\end{equation*}
\end{lemm}
Since the above lemma is the contrapositive of the following lemma which was utilized in \cite{Hi}, we omit the proof. 
\begin{lemm}[Lemma\ 4.1\ in \cite{Hi}]\label{modul_est}
We assume that $\sigma_{1}$, $\sigma_{2}$, $\sigma_{3} \in \R \backslash \{0\}$ satisfy $\widetilde{\kappa}\neq 0$ and $(\tau_{1},\xi_{1})$, $(\tau_{2}, \xi_{2})$, $(\tau_{3}, \xi_{3})\in \R\times \R^{2}$ satisfy $\tau_{1}+\tau_{2}+\tau_{3}=0$, $\xi_{1}+\xi_{2}+\xi_{3}=0$.  
If there exist $1\leq i,j\leq 3$ such that $|\xi_{i}|\ll |\xi_{j}|$, then we have
\begin{equation}\label{modulation_est}
\max_{1\leq j\leq 3}|\tau_{j}+\sigma_{j}|\xi_{j}|^{2}|
\gtrsim \max_{1\leq j\leq 3}|\xi_{j}|^{2}. 
\end{equation}
\end{lemm}
Lemma \ref{modul_est_2} suggests that if  $\displaystyle \max_{1\leq j\leq 3}|\tau_{j}+\sigma_{j}|\xi_{j}|^{2}|
\ll \max_{1\leq j\leq 3}|\xi_{j}|^{2}$ then we can assume
\begin{equation}
\max_{1 \leq j\leq 3} |\tau_{j}+\sigma_{j}|\xi_{j}|^{2}| \ll  
\min_{1\leq j\leq 3} |\xi_j|^2.
\end{equation}
We first introduce the angular frequency localization operators which were utilized in \cite{BHHT09}.
\begin{defn}[\cite{BHHT09}]
We define the angular decomposition of $\R^2$ in frequency.
We define a partition of unity in $\R$,
\begin{equation*}
1 = \sum_{j \in \Z} \omega_j, \qquad \omega_j (s) = \psi(s-j) \left( \sum_{k \in \Z} \psi (s-k) \right)^{-1}. 
\end{equation*}
For a dyadic number $A \geq 64$, we also define a partition of unity on the unit circle,
\begin{equation*}
1 = \sum_{j =0}^{A-1} \omega_j^A, \qquad \omega_j^A (\vartheta) = 
\omega_j \left( \frac{A\vartheta}{\pi} \right) + \omega_{j-A} \left( \frac{A\vartheta}{\pi} \right).
\end{equation*}
We observe that $\omega_j^A$ is supported in 
\begin{equation*}
\Theta_j^A = \left[\frac{\pi}{A} \, (j-2), \ \frac{\pi}{A} \, (j+2) \right] 
\cup \left[-\pi + \frac{\pi}{A} \, (j-2), \ - \pi +\frac{\pi}{A} \, (j+2) \right].
\end{equation*}
We now define the angular frequency localization operators $R_j^A$,
\begin{equation*}
\F_x (R_j^A f)(\xi) = \omega_j^A(\vartheta) \F_x f(\xi), \qquad \textnormal{where} \ \xi = |\xi| 
(\cos \vartheta, \sin \vartheta).
\end{equation*}
For any function $u  : \, \R \, \times \, \R^2 \, \to \C$, $(t,x) \mapsto u(t,x)$ we set 
$(R_j^A u ) (t, x) = (R_j^Au( t, \cdot)) (x)$. This operator localizes function in frequency to the set
\begin{equation*}
{\mathfrak{D}}_j^A = \{ (\tau, |\xi| \cos \vartheta, |\xi| \sin \vartheta) \in \R \times \R^2 
\, | \, \vartheta \in \Theta_j^A  \} .
\end{equation*}
Immediately, we can see
\begin{equation*}
u = \sum_{j=0}^{A-1} R_j^A u.
\end{equation*}
\end{defn}
The next lemma will be used to obtain Proposition~\ref{key_be} 
for the case $\widetilde{\theta}=0$
\begin{lemm}\label{angle_prop}
Let $N$, $L_1$, $L_2$, $L_3$, $A\in 2^{\N_0}$. 
We assume that $\sigma_{1}$, $\sigma_{2}$, $\sigma_{3} \in \R \backslash \{0\}$ satisfy $\widetilde{\theta}=0$ and $(\tau_{1},\xi_{1})$, $(\tau_{2}, \xi_{2})$, $(\tau_{3}, \xi_{3})\in \R\times \R^{2}$ satisfy $\tau_{1}+\tau_{2}+\tau_{3}=0$, $\xi_{1}+\xi_{2}+\xi_{3}=0$, 
$|\xi_i|\sim N_i$, $|\tau_i+\sigma_i|\xi_i|^2|\sim L_i$, 
 and $(\tau_i, \xi_i)\in {\mathfrak{D}}_{j_i}^A$ $(i=1,2,3)$
for some $j_{1}$, $j_{2}$, $j_3\in \{0,1,\cdots ,A-1\}$. 
If $N_1\sim N_2\sim N_3$, 
$\displaystyle L_{\max}:=\max_{1\le i\le 3}L_i\leq N_{\max}^2 A^{-2}$, 
and $A\gg 1$ hold, 
then we have $\min\{|j_1-j_2|, |A-(j_{1}-j_{2})|\}\lesssim 1$, 
$\min\{|j_2-j_3|, |A-(j_{2}-j_{3})|\}\lesssim 1$, and 
$\min\{|j_1-j_3|, |A-(j_{1}-j_{3})|\}\lesssim 1$.
\end{lemm}
\begin{proof}
Because $0=\widetilde{\theta}=\sigma_1\sigma_2\sigma_3(\frac{1}{\sigma_1}+\frac{1}{\sigma_2}+\frac{1}{\sigma_3})=\sigma_1\sigma_2+\sigma_2\sigma_3+\sigma_3\sigma_1$, 
we have
\[
(\sigma_1+\sigma_3)(\sigma_2+\sigma_3)
=\sigma_1\sigma_2+\sigma_2\sigma_3+\sigma_3\sigma_1+\sigma_3^2
=\sigma_3^2>0.
\]
We put $p:={\rm sgn}(\sigma_1+\sigma_3)={\rm sgn}(\sigma_2+\sigma_3)$, 
$q:={\rm sgn}(\sigma_3)$.
Let $\angle (\xi_{1},\xi_{2})\in [0,\pi]$ denotes 
the smaller angle between $\xi_{1}$ and $\xi_{2}$. 
Since
\[
\begin{split}
\frac{|\sigma_1+\sigma_3|^{\frac{1}{2}}|\sigma_2+\sigma_3|^{\frac{1}{2}}}{|\sigma_3|}
=\sqrt{1+\frac{\sigma_1\sigma_2\sigma_3}{\sigma_3^2}\left(\frac{1}{\sigma_1}+\frac{1}{\sigma_2}+\frac{1}{\sigma_3}\right)}=1, 
\end{split}
\]
we have
\[
\begin{split}
N_{\max}^2 A^{-2}&\geq L_{\max}\\
&\gtrsim |\sigma_1|\xi_1|^2+\sigma_2|\xi_2|^2+\sigma_3|\xi_1+\xi_2|^2|\\
&=|(\sigma_1+\sigma_3)|\xi_{1}|^2+(\sigma_2+\sigma_3)|\xi_{2}|^2
+2\sigma_3|\xi_{1}||\xi_{2}|\cos \angle (\xi_{1},\xi_{2})|\\
&=|p(|\sigma_1+\sigma_3|^{\frac{1}{2}}|\xi_{1}|-|\sigma_2+\sigma_3|^{\frac{1}{2}}
|\xi_{2}|)^2 \\
&\quad +2|\xi_{1}||\xi_{2}|
(p|\sigma_1+\sigma_3|^{\frac{1}{2}}|\sigma_2+\sigma_3|^{\frac{1}{2}}
+q|\sigma_3|\cos \angle (\xi_{1},\xi_{2}))|\\
&=|(|\sigma_1+\sigma_3|^{\frac{1}{2}}|\xi_{1}|-|\sigma_2+\sigma_3|^{\frac{1}{2}}
|\xi_{2}|)^2
+2|\sigma_3||\xi_{1}||\xi_{2}|
(1+pq\cos \angle (\xi_{1},\xi_{2}))|\\
&\ge 2|\sigma_3||\xi_{1}||\xi_{2}|
(1+pq\cos \angle (\xi_{1},\xi_{2}))
\end{split}
\]
Therefore we obtain
\[
\begin{split}
&1-\cos \angle (\xi_{1},\xi_{2})\lesssim A^{-2}\ \ {\rm if}\ (\sigma_1+\sigma_3)\sigma_3<0,\\
&1+\cos \angle (\xi_{1},\xi_{2})\lesssim  A^{-2} \ \ {\rm if}\ (\sigma_1+\sigma_3)\sigma_3>0.
\end{split}
\]
This implies
\[
\angle (\xi_{1},\xi_{2})\lesssim A^{-1}\ 
{\rm or}\ 
\pi -\angle (\xi_{1},\xi_{2})\lesssim A^{-1}.
\]
Therefore, 
we get $\min\{|j_1-j_2|, |A-(j_{1}-j_{2})|\}\lesssim 1$. 
By the same argument, we also get $\min\{|j_2-j_3|, |A-(j_{2}-j_{3})|\}\lesssim 1$ and 
$\min\{|j_1-j_3|, |A-(j_{1}-j_{3})|\}\lesssim 1$.
\end{proof}
Now we introduce the necessary bilinear estimates 
to obtain Proposition~\ref{key_be}  
for the case 
$\widetilde{\theta}<0$. 
\begin{thm}[Theorem\ 2.8\ in \cite{HK}]\label{thm-0.3}
We assume that $\sigma_{1}$, $\sigma_{2}$, $\sigma_{3} \in \R \backslash \{0\}$ satisfy $\widetilde{\kappa}\ne 0$ 
and $\widetilde{\theta}<0$.
Let $\displaystyle L_{\max} := \max_{1\leq j\leq 3} (L_1, L_2, L_3) \ll |\widetilde{\theta}| N_{\min}^2$, 
$A \geq 64$, and $|j_1 - j_2| \lesssim 1$. 
Then the following estimates holds:
\begin{align}
\begin{split}
\|Q_{L_3}^{-\sigma_3} P_{N_3}(R_{j_1}^A Q_{L_1}^{\sigma_1}P_{N_1}u_{1}\cdot 
R_{j_2}^A Q_{L_2}^{\sigma_2}P_{N_2}u_{2})\|_{L^{2}_{tx}} & \\
\lesssim A^{-\frac{1}{2}} L_1^{\frac{1}{2}}L_2^{\frac{1}{2}} 
\|R_{j_1}^A Q_{L_1}^{\sigma_1}P_{N_1}u_{1}\|_{L^2_{tx}} &  \|R_{j_2}^A Q_{L_2}^{\sigma_2}P_{N_2}u_{2}\|_{L^2_{tx}},
\end{split}\label{bilinear-12}\\
\begin{split}
\|R_{j_1}^A Q_{L_1}^{-\sigma_1} P_{N_1}(R_{j_2}^A Q_{L_2}^{\sigma_2}P_{N_2}u_{2}\cdot 
 Q_{L_3}^{\sigma_3}P_{N_3}u_{3})\|_{L^{2}_{tx}} & \\
\lesssim A^{-\frac{1}{2}} L_2^{\frac{1}{2}}L_3^{\frac{1}{2}} 
\|R_{j_2}^A Q_{L_2}^{\sigma_2}P_{N_2}u_{2}\|_{L^2_{tx}} &  \|Q_{L_3}^{\sigma_3}P_{N_3}u_{3}\|_{L^2_{tx}},
\end{split}\label{bilinear-23}\\
\begin{split}
\|R_{j_2}^A Q_{L_2}^{-\sigma_2} P_{N_2}( Q_{L_3}^{\sigma_3}P_{N_3}u_{3}\cdot 
 R_{j_1}^A Q_{L_1}^{\sigma_1}P_{N_1}u_{1})\|_{L^{2}_{tx}} & \\
\lesssim A^{-\frac{1}{2}} L_3^{\frac{1}{2}}L_1^{\frac{1}{2}} \|Q_{L_3}^{\sigma_3}P_{N_3}u_{3}\|_{L^2_{tx}} & 
\|R_{j_1}^A Q_{L_1}^{\sigma_1}P_{N_1}u_{1}\|_{L^2_{tx}}.
\end{split}\label{bilinear-31}
\end{align}
\end{thm}
\begin{prop}[Proposition\ 2.9\ in \cite{HK}]\label{thm2.6}
We assume that $\sigma_{1}$, $\sigma_{2}$, $\sigma_{3} \in \R \backslash \{0\}$ satisfy $\widetilde{\kappa}\ne 0$ 
and $\widetilde{\theta}<0$.
Let $L_{\textnormal{max}} \ll |\widetilde{\theta}| 
N_{\min}^2$ 
and $ 64 \leq A \leq N_{\textnormal{max}}$, \ $16 \leq |j_1 - j_2 |\leq 32$. 
Then the following estimate holds:
\begin{equation}
\begin{split}
\|Q_{L_3}^{-\sigma_3} P_{N_3}(R_{j_1}^A Q_{L_1}^{\sigma_1}P_{N_1}u_{1}\cdot 
R_{j_2}^A Q_{L_2}^{\sigma_2}P_{N_2}u_{2})\|_{L^{2}_{tx}} & \\
\lesssim A^{\frac{1}{2}} N_1^{-1} L_1^{\frac{1}{2}}L_2^{\frac{1}{2}} L_3^{\frac{1}{2}} 
\|R_{j_1}^A Q_{L_1}^{\sigma_1}P_{N_1}u_{1}\|_{L^2_{tx}} &  \|R_{j_2}^A Q_{L_2}^{\sigma_2}P_{N_2}u_{2}\|_{L^2_{tx}}.
\end{split}\label{0609}
\end{equation}
\end{prop}
\subsection{Proof of Proposition~\ref{key_be}}
By the duality argument, we have
\[
\begin{split}
\||\nabla |(uv)\|_{X^{s,-b'}_{-\sigma_3}}
&\lesssim 
\sup_{\|w\|_{X^{-s,b'}_{\sigma_3}}=1}
\left|\int |\nabla|(uv)wdxdt\right|,\\ 
\|(\Delta U)v\|_{X^{s,-b'}_{-\sigma_3}}
&\lesssim 
\sup_{\|w\|_{X^{-s,b'}_{\sigma_3}}=1}
\left|\int (\Delta U)vwdxdt\right|\\
&\le \sup_{\|w\|_{X^{-s,b'}_{\sigma_3}}=1}
\left(
\left|\int \partial_1(\partial_1U)vwdxdt\right|
+\left|\int \partial_2(\partial_2U)vwdxdt\right|
\right), 
\end{split}
\]
where we used $(Q_{L_3}^{-\sigma_3}f,\overline{g})_{L^2_{tx}}=(f,\overline{Q_{L_3}^{\sigma_3}g})_{L^2_{tx}}$. 
Since $|\nabla |(uv)$ and $(\Delta U)v$ are radial with respect to $x$, 
we can assume $w$ is also radial with respect to $x$. 
Therefore, to obtain (\ref{be_111}), it suffices to show that
\begin{equation}\label{tri_est_000}
\begin{split}
&\sum_{N_1,N_2,N_3\ge 1}\sum_{L_1,L_2, L_3 \ge 1} N_{\max}
\left|\int u_{N_1,L_1}v_{N_2,L_2}w_{N_3,L_3}dxdt\right|\\
&\lesssim 
\|u\|_{X^{s,b'}_{\sigma_1}} \|v\|_{X^{s,b'}_{\sigma_2}} \|w\|_{X^{-s, b'}_{\sigma_3}}
\end{split}
\end{equation}
for the radial functions $u$, $v$, and $w$, where we put
\[
u_{N_1,L_1}:=Q_{L_1}^{\sigma_1}P_{N_1}u,\ 
v_{N_2,L_2}:=Q_{L_2}^{\sigma_2}P_{N_2}v,\ 
w_{N_3,L_3}:=Q_{L_3}^{\sigma_3}P_{N_3}w
\]
and used $(Q_{L_3}^{-\sigma_3}f,\overline{g})_{L^2_{tx}}=(f,\overline{Q_{L_3}^{\sigma_3}g})_{L^2_{tx}}$. 
%While (\ref{be_222}) can be obtained by applying (\ref{tri_est_000}) 
%as $u=\partial_jU$ $(j=1,2)$. 
By Plancherel's theorem, we have
\[
\begin{split}
&\left|\int u_{N_1,L_1}v_{N_2,L_2}w_{N_3,L_3}dxdt\right|\\
&\sim 
\left|\int_{\substack{\xi_1+\xi_2+\xi_3=0\\ \tau_1+\tau_2+\tau_3=0}}
\F_{tx}[u_{N_1,L_1}](\tau_1,\xi_1)\F_{tx}[v_{N_2,L_2}](\tau_2,\xi_2)\F_{tx}[w_{N_3,L_3}](\tau_3,\xi_3)\right|.
\end{split}
\]
We only consider the case $N_1 \lesssim N_2 \sim N_3$, 
because the remaining cases $N_2 \lesssim N_3 \sim N_1$ and 
$N_3 \lesssim N_1 \sim N_2$ can be shown similarly. 
It suffices to show that
\begin{equation}\label{desired_est}
\begin{split}
&N_{2}
\left|\int u_{N_1,L_1}v_{N_2,L_2}w_{N_3,L_3}dxdt\right|\\
&\lesssim
\left( \frac{N_{1}}{N_{2}} \right)^{\epsilon}N_{1}^s(L_1L_2L_3)^{c}
\|u_{N_1,L_1}\|_{L^2_{tx}}\|v_{N_2,L_2}\|_{L^2_{tx}}\|w_{N_3,L_3}\|_{L^2_{tx}}
\end{split}
\end{equation}
for some $b'\in (0,\frac{1}{2})$, $c\in (0,b')$, and $\epsilon >0$. 
Indeed, from (\ref{desired_est}) and the Cauchy-Schwarz inequality, 
we obtain 
\[
\begin{split}
&
\sum_{N_1 \lesssim N_2 \sim N_3}\sum_{L_1,L_2, L_3 \ge 1} N_{2}
\left|\int u_{N_1,L_1}v_{N_2,L_2}w_{N_3,L_3}dxdt\right|\\
&\lesssim 
 \sum_{N_1 \lesssim N_2 \sim N_3}\sum_{L_1,L_2, L_3\ge 1}
\left( \frac{N_{1}}{N_{2}} \right)^{\epsilon}N_{1}^s(L_1L_2L_3)^{c}
\|u_{N_1,L_1}\|_{L^2_{tx}}\|v_{N_2,L_2}\|_{L^2_{tx}}\|w_{N_3,L_3}\|_{L^2_{tx}}\\
&\lesssim  \sum_{N_3}\sum_{ N_2 \sim N_3}
\left( \sum_{N_1\lesssim N_2}N_1^{s+\e} N_2^{-\e} \sum_{L_1 \ge 1}L_1^{c}\|u_{N_1,L_1}\|_{L^2_{tx}}\right) \\
& \quad \times \sum_{L_2\ge 1}L_2^{c}\|v_{N_2,L_2}\|_{L^2_{tx}}\sum_{L_3\ge 1}L_3^{c}\|w_{N_3,L_3}\|_{L^2_{tx}} \\
&\lesssim \|u\|_{X^{s,b'}_{\sigma_1}} 
\sum_{N_3}\sum_{ N_2 \sim N_3} \left( N_2^{2s} \sum_{L_2\ge 1}L_2^{2b'}\|v_{N_2,L_2}\|_{L^2_{tx}}^2 \right)^{\frac{1}{2}}
\left( N_3^{-2s}  \sum_{L_3\ge 1}L_3^{2b'}\|w_{N_3,L_3}\|_{L^2_{tx}}^2 \right)^{\frac{1}{2}}\\
& \lesssim \|u\|_{X^{s,b'}_{\sigma_1}} \|v\|_{X^{s,b'}_{\sigma_2}} \|w\|_{X^{-s, b'}_{\sigma_3}}
\end{split}
\]
%Hence, we focus on (\ref{desired_est}) for $N_1\lesssim N_2\sim N_3$. 
We put 
$\displaystyle L_{\max} := \max_{1\leq j\leq 3} (L_1, L_2, L_3)$.\\
\kuuhaku \\
\underline{Case\ 1:\ High modulation, $\displaystyle L_{\max}\gtrsim N_{\max}^2$}

In this case, the radial condition is not needed. 
We assume $L_1\gtrsim N_{\max}^2\sim N_2^2$. 
By the Cauchy-Schwarz inequality and (\ref{L2be_est_2}), 
we have
\[
\begin{split}
&\left|\int u_{N_1,L_1}v_{N_2,L_2}w_{N_3,L_3}dxdt\right|\\
&\lesssim \|u_{N_1,L_1}\|_{L^2_{tx}}\|P_{N_1}(v_{N_2,L_2}w_{N_3,L_3})\|_{L^2_{tx}}\\
&\lesssim N_1^{4\delta}
\left(\frac{N_1}{N_2}\right)^{\frac{1}{2}- 2\delta}L_2^{c}L_3^{c}
\|u_{N_1,L_1}\|_{L^2_{tx}}\|v_{N_2,L_2}\|_{L^2_{tx}}\|w_{N_3,L_3}\|_{L^2_{tx}},
\end{split}
\]
where $\delta := \frac{1}{2}-c$.
Therefore, we obtain
\[
\begin{split}
&N_{2}
\left|\int u_{N_1,L_1}v_{N_2,L_2}w_{N_3,L_3}dxdt\right|\\
&\lesssim N_1^{\frac{1}{2}+2\delta} N_2^{\frac{1}{2}-2c+2\delta}
(L_1L_2L_3)^{c}
\|u_{N_1,L_1}\|_{L^2_{tx}}\|v_{N_2,L_2}\|_{L^2_{tx}}\|w_{N_3,L_3}\|_{L^2_{tx}}.
\end{split}
\]
Thus, it suffices to show that
\begin{equation}\label{wait_esti}
N_1^{\frac{1}{2}+2\delta} N_2^{\frac{1}{2}-2c+2\delta}
\lesssim \left( \frac{N_{1}}{N_{2}} \right)^{\epsilon}N_{1}^s.
\end{equation}
Since $\delta =\frac{1}{2}-c$, we have
\[
\begin{split}
N_1^{\frac{1}{2}+2\delta} N_2^{\frac{1}{2}-2c+2\delta}
&=N_1^{\frac{3}{2}- 2c} N_2^{\frac{3}{2}- 4c} \\
&\sim N_1^{3- 6c-s}
\left( \frac{N_{1}}{N_{2}} \right)^{4c- \frac{3}{2} }N_{1}^s.
\end{split}
\]
Therefore, by choosing $b'$ and $c$ as
$\max\{\frac{3-s}{6},\frac{3}{8}\}<c<b'<\frac{1}{2}$
for $s>0$, 
we get (\ref{wait_esti}). \\
\begin{comment}
If $d=2$, then we obtain
\[
\begin{split}
 N_1^{\frac{d+4}{2}- 6 b' -s}
\le N_1^{-6(b'-\frac{5}{12})}
\end{split}
\]
for $s\ge \frac{1}{2}$. 
Therefore, by choosing $b'\in [\frac{5}{12}, \frac{1}{2})$, 
we get (\ref{wait_esti}). 
While if $d=3$, then we obtain 
\[
\begin{split}
 N_1^{\frac{d+4}{2}- 6 b' -s}
= N_1^{-\left( \left( s-\frac{1}{2} \right) -6 \left( \frac{1}{2} - b' \right) \right) }.
\end{split}
\]
Therefore, by choosing $b'\in (\frac{1}{2}-\frac{1}{6}(s-\frac{1}{2}),\frac{1}{2})$ 
for $s>\frac{1}{2}$, we get (\ref{wait_esti}). 

The proofs for the cases $L_2\gtrsim N_{\max}^2$ and $L_3\gtrsim N_{\max}^2$ are quite same. We omit them.\\
\end{comment}
\kuuhaku \\
\underline{Case\ 2:\ Low modulation, $\displaystyle L_{\max}\ll N_{\max}^2$}

By Lemma \ref{modul_est_2}, we can assume $N_1 \sim N_2 \sim N_3$ thanks to $\displaystyle L_{\max}\ll N_{\max}^2$. 
We assume $L_{\textnormal{max}} = L_3$ for 
simplicity. The other cases can be treated similarly. \\

\noindent\textbf{$\circ$ The case $\widetilde{\theta}=0$}\\
Let $A := L_{\max}^{-\frac{1}{2}} N_{\max}\sim L_{3}^{-\frac{1}{2}} N_{1}$. We decompose $\R^3 \cross \R^3\cross \R^3$ as follows:
\begin{equation*}
\R^3 \cross \R^3\cross \R^3 = \bigcup_{0 \leq j_1,j_2,j_3 \leq A -1} 
{\mathfrak{D}}_{j_1}^A \cross {\mathfrak{D}}_{j_2}^A\cross {\mathfrak{D}}_{j_3}^A.
\end{equation*}
Since $L_{\textnormal{max}} \leq N_{\max}^2 (L_{\max}^{-\frac{1}{2}} N_{\max})^{-2} = N_{\max}^2 A^{-2}$, by Lemma~\ref{angle_prop}, we can write
\begin{align*}
 &\left|\int u_{N_1,L_1}v_{N_2,L_2}w_{N_3,L_3}dxdt\right|\\
&\leq \sum_{j_1=0}^{A-1}\sum_{j_2\in J(j_1)} 
\sum_{j_3\in J(j_1)} 
\left|\int u_{N_1,L_1, j_1}v_{N_2,L_2, j_2}w_{N_3,L_3,j_3}dxdt\right| 
\end{align*}
with $u_{N_1,L_1, j_1} := R_{j_1}^A u_{N_1, L_1}$, 
$v_{N_2,L_2, j_2} := R_{j_2}^A v_{N_2, L_2}$ and
$w_{N_3,L_3, j_3} := R_{j_3}^A v_{N_3, L_3}$, where
\[
J(j_1):=\{j\in \{0,1,\cdots ,A-1\}|\min\{|j_1-j|, |A-(j_1-j)|\}\lesssim 1\}. 
\]
We note that $\# J(j_1)\lesssim 1$. 
By using the H\"older inequality and Corollary~\ref{Bo_Stri} 
with $p=q=4$, we get
\begin{align*}
& \sum_{j_1=0}^{A-1}\sum_{j_2\in J(j_1)} 
\sum_{j_3\in J(j_1)} 
\left|\int u_{N_1,L_1, j_1}v_{N_2,L_2, j_2}w_{N_3,L_3,j_3}dxdt\right|  \\
& \lesssim  \sum_{j_1=0}^{A-1}\sum_{j_2\in J(j_1)} 
\sum_{j_3\in J(j_1)} 
\|u_{N_1,L_1,j_1}\|_{L^4_{tx}}\|v_{N_2,L_2,j_2}\|_{L^4_{tx}}\|w_{N_3,L_3,j_3} \|_{{L^2_{t x}}}  
\\
& \lesssim 
AL_1^{\frac{1}{2}}L_2^{\frac{1}{2}}
\sup_{j_1}\|u_{N_1,L_1,j_1}\|_{L^2_{tx}}
\sup_{j_2}\|v_{N_2,L_2,j_2}\|_{L^2_{tx}}
\sup_{j_3}\|w_{N_3,L_3,j_3}\|_{L^2_{tx}}.
\end{align*}
Since $u$, $v$, and $w$ are radial respect to $x$, we have
\begin{equation}\label{rad_l2est}
\begin{split}
&\|u_{N_1,L_1,j_1}\|_{L^2_{tx}}\lesssim A^{-\frac{1}{2}}\|u_{N_1,L_1}\|_{L^2_{tx}},\ 
\|v_{N_2,L_2,j_2}\|_{L^2_{tx}}\lesssim A^{-\frac{1}{2}}\|v_{N_2,L_2}\|_{L^2_{tx}},\\ 
&\|w_{N_3,L_3,j_3}\|_{L^2_{tx}}\lesssim A^{-\frac{1}{2}}\|w_{N_3,L_3}\|_{L^2_{tx}}. 
\end{split}
\end{equation}
Therefore, we obtain 
\[
\begin{split}
&N_2\left|\int u_{N_1,L_1}v_{N_2,L_2}w_{N_3,L_3}dxdt\right|\\
&\lesssim 
N_2A^{-\frac{1}{2}}L_1^{\frac{1}{2}}L_2^{\frac{1}{2}}
\|u_{N_1,L_1}\|_{L^2_{tx}}
\|v_{N_2,L_2}\|_{L^2_{tx}}
\|w_{N_3,L_3}\|_{L^2_{tx}}\\
&\sim N_1^{\frac{1}{2}}L_1^{\frac{1}{2}}L_2^{\frac{1}{2}}L_3^{\frac{1}{4}}
\|u_{N_1,L_1}\|_{L^2_{tx}}
\|v_{N_2,L_2}\|_{L^2_{tx}}
\|w_{N_3,L_3}\|_{L^2_{tx}}\\
&\lesssim N_1^{\frac{1}{2}}(L_1L_2L_3)^{\frac{5}{12}}
\|u_{N_1,L_1}\|_{L^2_{tx}}
\|v_{N_2,L_2}\|_{L^2_{tx}}
\|w_{N_3,L_3}\|_{L^2_{tx}}
\end{split}
\]
This estimate gives the desired estimate \eqref{desired_est} for $s\ge \frac{1}{2}$ 
by choosing $b'$ and $c$ as $\frac{5}{12}\le c<b'<\frac{1}{2}$. \\

\noindent\textbf{$\circ$ The case $\widetilde{\theta}<0$}\\
We decompose $\R^3 \cross \R^3$ as follows:
\begin{equation*}
\R^3 \cross \R^3 =  \left(\bigcup_{\tiny{\substack{0 \leq j_1,j_2 \leq N_1 -1\\|j_1 - j_2|\leq 16}}} 
{\mathfrak{D}}_{j_1}^{N_1} \cross {\mathfrak{D}}_{j_2}^{N_1}\right)
 \cup 
\left(\bigcup_{64 \leq A \leq N_1} \ \bigcup_{\tiny{\substack{0 \leq j_1,j_2 \leq A -1\\ 16 \leq |j_1 - j_2|\leq 32}}} 
{\mathfrak{D}}_{j_1}^A \cross {\mathfrak{D}}_{j_2}^A\right).
\end{equation*}
We can write
\begin{align*}
 \left|\int u_{N_1,L_1}v_{N_2,L_2}w_{N_3,L_3}dxdt\right| 
 \leq \sum_{{\tiny{\substack{A=N_1\\0 \leq j_1,j_2 \leq N_1 -1\\|j_1 - j_2|\leq 16}}}} \sum_{j_3\in J(j_1)} 
\left|\int u_{N_1,L_1, j_1}v_{N_2,L_2, j_2}w_{N_3,L_3,j_3}dxdt\right| & \\ + 
\sum_{64 \leq A \leq N_1}  \sum_{{\tiny{\substack{0 \leq j_1,j_2 \leq A-1\\16\le |j_1 - j_2|\leq 32}}}} 
\sum_{j_3\in J(j_1)} 
\left|\int u_{N_1,L_1, j_1}v_{N_2,L_2, j_2}w_{N_3,L_3,j_3}dxdt\right|&. 
\end{align*}
For the former term, by using the H\"older inequality, Theorem \ref{thm-0.3}, 
and (\ref{rad_l2est}), we get
\begin{align*}
& \sum_{{\tiny{\substack{A=N_1\\0 \leq j_1,j_2 \leq N_1 -1\\|j_1 - j_2|\leq 16}}}} \sum_{j_3\in J(j_1)} 
\left|\int u_{N_1,L_1, j_1}v_{N_2,L_2, j_2}w_{N_3,L_3,j_3}dxdt\right| \\
& \lesssim  
\sum_{{\tiny{\substack{A=N_1\\ 0 \leq j_1,j_2 \leq N_1 -1\\|j_1 - j_2|\leq 16}}}}
\|Q_{L_3}^{-\sigma_3} P_{N_3} ( u_{N_1,L_1, j_1} v_{N_2,L_2, j_2})\|_{L^{2}_{tx}}
\sum_{j_3\in J(j_1)} \|w_{N_3,L_3, j_3} \|_{{L^2_{t x}}}  \\
& \lesssim 
N_1^{-1} L_1^{\frac{1}{2}}L_2^{\frac{1}{2}} \|w_{N_3,L_3} \|_{{L^2_{t x}}}
\sum_{{\tiny{\substack{A=N_1\\ 0 \leq j_1,j_2 \leq N_1 -1\\|j_1 - j_2|\leq 16}}}}
\|u_{N_1,L_1, j_1}\|_{L^2_{tx}}   \|v_{N_2,L_2, j_2}\|_{L^2_{tx}}\\
& \lesssim  N_1^{-1}  L_1^{\frac{1}{2}}L_2^{\frac{1}{2}} \|u_{N_1,L_1}\|_{L^2_{tx}}\|v_{N_2,L_2}\|_{L^2_{tx}}\|w_{N_3,L_3}\|_{L^2_{tx}}\\
&\lesssim N_1^{-1}  (L_1L_2L_3)^{\frac{1}{3}} \|u_{N_1,L_1}\|_{L^2_{tx}}\|v_{N_2,L_2}\|_{L^2_{tx}}\|w_{N_3,L_3}\|_{L^2_{tx}}.
\end{align*}
For the latter term, by using Proposition \ref{thm2.6}, (\ref{rad_l2est}), 
and $L_1L_2L_3\lesssim N_{1}^6$ that we get
\begin{align*}
& \sum_{64 \leq A \leq N_1} \sum_{{\tiny{\substack{0 \leq j_1,j_2 \leq A-1\\16\le |j_1 - j_2|\leq 32}}}} \sum_{j_3\in J(j_1)} 
\left|\int u_{N_1,L_1, j_1}v_{N_2,L_2, j_2}w_{N_3,L_3,j_3}dxdt\right| \\
& \lesssim  
\sum_{64 \leq A \leq N_1} \sum_{{\tiny{\substack{0 \leq j_1,j_2 \leq A-1\\16\le |j_1 - j_2|\leq 32}}}} 
\|Q_{L_3}^{-\sigma_3} P_{N_3} ( u_{N_1,L_1, j_1} v_{N_2,L_2, j_2})\|_{L^{2}_{tx}} 
\sum_{j_3\in J(j_1)} \|w_{N_3,L_3, j_3} \|_{{L^2_{t x}}}\\
& \lesssim \|w_{N_3,L_3} \|_{{L^2_{t x}}}
\sum_{64 \leq A \leq N_1}  
 N_1^{-1} ( L_1 L_2 L_3)^\frac{1}{2}
 \sum_{{\tiny{\substack{0 \leq j_1,j_2 \leq A-1\\16\le |j_1 - j_2|\leq 32}}}}  
\|u_{N_1,L_1, j_1}\|_{L^2_{tx}}   \|v_{N_2,L_2, j_2}\|_{L^2_{tx}}\\
&  \lesssim (\log{N_1}) N_1^{-1}  ( L_1 L_2 L_3)^{\frac{1}{2}} \|u_{N_1,L_1}\|_{L^2_{tx}}\|v_{N_2,L_2}\|_{L^2_{tx}}\|w_{N_3,L_3}\|_{L^2_{tx}}\\
&\lesssim (\log{N_1}) N_1^{2-6c}  ( L_1 L_2 L_3)^{c} \|u_{N_1,L_1}\|_{L^2_{tx}}\|v_{N_2,L_2}\|_{L^2_{tx}}\|w_{N_3,L_3}\|_{L^2_{tx}}.
\end{align*}
The above two estimates give the desired estimate \eqref{desired_est} for $s>0$ 
by choosing $b'$ and $c$ as $\max\{\frac{3-s}{6},\frac{1}{3}\}< c<b'<\frac{1}{2}$. 
\hspace{\fill}
$\square$
%
%
%
%%%%%%%%%%%%%%%%%%%%%%%%%%%%%%%%%%%%%%%%%%%%%%%%%%%%%%%%%%%%%%%%%%%%%%%%%%%%%%%%%
%%%%%%%%%%%%%%%%%%%%%%%%%%%%%%%%%%%%%%%%%%%%%%%%%%%%%%%%%%%%%%%%%%%%%%%%%%%%%%%%%
%%%%%%%%%%%%%%%%%%%%%%%%%%%%%%%%%%%%  Section 3   %%%%%%%%%%%%%%%%%%%%%%%%%%%%%%%%%%
%%%%%%%%%%%%%%%%%%%%%%%%%%%%%%%%%%%%%%%%%%%%%%%%%%%%%%%%%%%%%%%%%%%%%%%%%%%%%%%%%
%%%%%%%%%%%%%%%%%%%%%%%%%%%%%%%%%%%%%%%%%%%%%%%%%%%%%%%%%%%%%%%%%%%%%%%%%%%%%%%%%
%
\section{Proof of the well-posedness}
In this section, we prove Theorems~\ref{wellposed_1} and ~\ref{wellposed_2}.
For a Banach space $H$ and $r>0$, we define $B_r(H):=\{ f\in H \,|\, \|f\|_H \le r \}$. 
Furthermore, we define $\mathcal{X}^{s,b}_{T}$ as
\[
\mathcal{X}^s_T:= (X^{s,b}_{\alpha,{\rm rad},T})^2\times (X^{s,b}_{\beta,{\rm rad},T})^2\times \widetilde{X}^{s+1,b}_{\gamma,{\rm rad},T},
\] 
where $X^{s,b}_{\alpha,{\rm rad},T}$ and $X^{s,b}_{\beta,{\rm rad},T}$
are the time localized spaces
defined by
\[
X^{s,b}_{\sigma,{\rm rad},T}:=\{u|_{[0,T]}|\ u\in X^{s,b}_{\sigma,{\rm rad}}\}
\]
with the norm
\[
\|u\|_{X^{s,b}_{\sigma, T}}:=\inf \{\|v\|_{X^{s,b}_{\sigma,T}}|\ v\in X^{s,b}_{\sigma ,{\rm rad}},\ v|_{[0,T]}=u|_{[0,T]}\}.
\]
Also, $\widetilde{X}^{s+1,b}_{\gamma,{\rm rad},T}$ is defined by the same way. 
Now, we restate Theorems~\ref{wellposed_1} for $d=2$ more precisely. 
\begin{thm}\label{wellposed_re}
Let $s\ge \frac{1}{2}$ if $\theta =0$ and $s>0$ if $\theta <0$. 
For any $r>0$ and for all initial data $(u_{0}, v_{0}, [W_{0}])\in B_r(\mathcal{H}^s(\R^2))$, there exist $T=T(r)>0$ and a solution
$(u,v,[W])\in \mathcal{X}^{s,b}_T$ 
to the system {\rm (\ref{NLS_rad})} on $[0, T]$  for suitable $b>\frac{1}{2}$. 
Such solution is unique in $B_R(\mathcal{X}^s_T)$ for some $R>0$. 
Moreover, the flow map
\[
S:B_{r}(\mathcal{H}^s(\R^2))\ni (u_{0},v_{0},[W_{0}])\mapsto (u,v,[W])\in \mathcal{X}^s_T
\]
is Lipschitz continuous.
\end{thm}
\begin{rem}
Since $X^{s,b}_T\hookrightarrow C([0,T];H^s(\R^2))$ holds 
for $b>\frac{1}{2}$, 
we have $\mathcal{X}^{s,b}_T\hookrightarrow C([0,T];\mathcal{H}^s(\R^2))$.
\end{rem}
To prove Theorem~\ref{wellposed_re}, 
we give the linear estimate. 
\begin{prop}\label{linear_est}
Let $s\in \R$, $\sigma\in \R\backslash \{0\}$, $b\in (\frac{1}{2},1]$, 
$b'\in [0,1-b]$ and 
$0<T\le 1$. 
\begin{enumerate}
\item[(1)] There exists $C_1>0$ such that for any $\varphi \in H^s(\R^2)$, we have
\[
\|e^{it\sigma \Delta}\varphi\|_{X^{s,b}_{\sigma,T}}
\le C_1\|\varphi\|_{H^s}.
\]
\item[(2)] There exists $C_2>0$ such that for any $F \in X^{s,-b'}_{\sigma,T}$, we have
\[
\left\|\int_{0}^{t}e^{i(t-t')\sigma \Delta}F(t')dt'\right\|_{X^{s,b}_{\sigma,T}}
\le C_2T^{1-b'-b}\|F\|_{X^{s,-b'}_{\sigma,T}}.
\]
\item[(3)] There exists $C_3>0$ such that for any $u \in X^{s,b}_{\sigma,T}$, we have
\[
\|u\|_{X^{s,b'}_{\sigma, T}}\le C_3T^{b-b'}\|u\|_{X^{s,b}_{\sigma, T}}. 
\]
\end{enumerate}
\end{prop}
For the proof of Proposition~\ref{linear_est}, 
see Lemma\ 2.1 and 3.1 in \cite{GTV97}. 

We define the map $\Phi(u,v,[W])=(\Phi_{\alpha, u_{0}}^{(1)}([W], v), \Phi_{\beta, v_{0}}^{(1)}([\overline{W}], v), [\Phi_{\gamma, [W_{0}]}^{(2)}(u, \overline{v}))])$ as
\[
\begin{split}
\Phi_{\sigma, \varphi}^{(1)}([f],g)(t)&:=e^{it\sigma \Delta}\varphi 
-i\int_{0}^{t}e^{i(t-t')\sigma \Delta}(\Delta f(t'))g(t')dt',\\
\Phi_{\sigma, [\varphi]}^{(2)}(f,g)(t)&:=e^{it\sigma \Delta}\varphi 
+i\int_{0}^{t}e^{i(t-t')\sigma \Delta} (f(t')\cdot g(t'))dt'.
\end{split}
\] 
To prove the existence of the solution of (\ref{NLS_sys}), we prove that $\Phi$ is a contraction map 
on $B_R(\mathcal{X}^s_T)$ for some $R>0$ and $T>0$. 
For a vector valued function $f=(f_1,f_2)$, $\|f\|_{H^s}$ and 
$\|f\|_{X^{s,b}_T}$ 
denote $\|f_1\|_{H^s}+\|f_2\|_{H^s}$ and $\|f_1\|_{X^{s,b}_T}+\|f_2\|_{X^{s,b}_T}$, respectively. 
\begin{comment}
We note that  
\[
\|W_0\|_{\dot{H}^{s+1}}+\|W_0\|_{\dot{H}^1}\sim \|\nabla W_0\|_{H^s},\ \ 
\|W\|_{\dot{X}^{s+1,\frac{1}{2},1}_T}+\|W\|_{\dot{X}^{1,\frac{1}{2},1}_T}
\sim \|\nabla W\|_{X^{s,\frac{1}{2},1}_T}
\]
hold for $W_0\in \dot{H}^{s+1}\cap \dot{H}^1$ and
$W\in \dot{X}^{s+1,\frac{1}{2},1}_T\cap \dot{X}^{1,\frac{1}{2},1}_T$.
\end{comment}
%
\begin{proof}[\rm{\bf{Proof of Theorem~\ref{wellposed_re}.}}]
We choose $b>\frac{1}{2}$ as $b=1-b'$, 
where $b'$ is as in Proposition~\ref{key_be}. 
Let $(u_{0}$, $v_{0}$, $[W_{0}])\in B_{r}(\mathcal{H}^s(\R^2))$ be given. 
By Proposition~\ref{key_be} with $(\sigma_1,\sigma_2,\sigma_3)\in \{(\beta, \gamma, -\alpha), 
(-\gamma, \alpha, -\beta), (\alpha, -\beta, -\gamma)\}$ 
and Proposition~\ref{linear_est} with $\sigma \in \{\alpha, \beta, \gamma\}$, 
there exist constants $C_1$, $C_2$, $C_3>0$ such that for any $(u,v,[W])\in B_R(\mathcal{X}^s_T)$, 
we have
\[
\begin{split}
\|\Phi^{(1)}_{\alpha ,u_{0}}([W], v)\|_{X^{s,b}_{\alpha,T}}&
\leq C_1\|u_{0}\|_{H^s} +CC_2C_3^2T^{4b-2}
\|[W]\|_{\widetilde{X}^{s+1,b}_{\gamma,T}}
\|v\|_{X^{s,b}_{\beta,T}} \\
&\leq C_1r+CC_2C_3^2T^{4b-2}R^2,\\
\|\Phi^{(1)}_{\beta ,v_{0}}([\overline{W}], u)\|_{X^{s,b}_{\beta,T}}&
\leq C_1\|v_{0}\|_{H^s} +CC_2C_3^2T^{4b-2}\|[W]\|_{\widetilde{X}^{s+1,b}_{\gamma,T}}
\|u\|_{X^{s,b}_{\alpha,T}}\\
&\leq C_1r+CC_2C_3^2T^{4b-2}R^2,\\
\|[\Phi^{(2)}_{\gamma ,[W_{0}]}(u, \overline{v})]\|_{\widetilde{X}^{s+1,b}_{\gamma,T}}&
\leq C_1\|[W_{0}]\|_{\widetilde{H}^{s+1}} +CC_2C_3^2T^{4b-2}\|u\|_{X^{s,b}_{\alpha,T}}
\|v\|_{X^{s,b}_{\beta,T}}\\
&\leq C_1r+CC_2C_3^2T^{4b-2}R^2.
\end{split}
\]
Similarly,
\[
\begin{split}
&\|\Phi^{(1)}_{\alpha ,u_{0}}([W], v)-\Phi^{(1)}_{\alpha ,u_{0}}([W'], v')\|_{X^{s,b}_{\alpha,T}} \\
&\quad \leq CC_2C_3^2T^{4b-2}R\left( \|[W]-[W']\|_{\widetilde{X}^{s+1,b}_{\gamma,T}}+\|v-v'\|_{X^{s,b}_{\beta,T}}\right),\\
&\|\Phi^{(1)}_{\beta ,v_{0}}([\overline{W}], u)-\Phi^{(1)}_{\beta ,v_{0}}([\overline{W'}], u')\|_{X^{s,b}_{\beta,T}} \\
&\quad \leq CC_2C_3^2T^{4b-2}R\left( \|[W]-[W']\|_{\widetilde{X}^{s+1,b}_{\gamma,T}}+\|u-u'\|_{X^{s,b}_{\alpha,T}}\right),\\
&\|[\Phi^{(2)}_{\gamma ,[W_{0}]}(u, \overline{v})]-[\Phi^{(2)}_{\gamma ,[W_{0}]}(u', \overline{v'})]\|_{\widetilde{X}^{s+1,b}_{\gamma,T}} \\
&\quad \leq CC_2C_3^2T^{4b-2}R\left( \|u-u'\|_{X^{s,b}_{\alpha,T}}+\|v-v'\|_{X^{s,b}_{\beta,T}}\right).
\end{split}
\]
Therefore if we choose $R>0$ and $T>0$ as
\[
R=6C_1r,\ CC_2C_3^2T^{4b-2}R\le \frac{1}{4}
\]
then $\Phi$ is a contraction map on $B_R(\mathcal{X}^s_T)$. 
This implies the existence of the solution of the system (\ref{NLS_sys}) and the uniqueness in the ball $B_R(\mathcal{X}^s_T)$. 
The Lipschitz-continuity of the flow map is also proved by similar argument. 
\end{proof} 
Next, to prove Theorem~\ref{wellposed_2}, 
we justify the existence of a scalar potential of $w\in (H^s(\R^2))^2$. 
Let $\F_1$ and $\F_2$ denote the 
Fourier transform with respect to 
the first component and the second component, respectively. 
We note that $\F_1^{-1}\F_2^{-1}=\F_2^{-1}\F_1^{-1}=\F_x^{-1}$ 
(and also $\F_1\F_2=\F_2\F_1=\F_x$) holds on $L^2(\R^2)$. 
\begin{prop}\label{ex_sc}
Let $s>\frac{1}{2}$ and $w=(w_1,w_2)\in (H^s(\R^2))^2$. 
If $w_1$ and $w_2$ satisfy
\[
\xi_2\widehat{w_1}(\xi)-\xi_1\widehat{w_2}(\xi) =0\ \ {\rm a.e.}\ \xi=(\xi_1,\xi_2)\in \R^2,
\]
then there exists $W\in L^1_{\rm loc}(\R^2)$ $(\subset \s'(\R^2))$ such that
\[
\nabla W(x)=w(x)\ \ {\rm a.e.}\ x=(x_1,x_2)\in \R^2. 
\]
\end{prop}
To obtain Proposition~\ref{ex_sc}, we use the next lemma. 
\begin{lemm}\label{l1_lem}
Let $s>\frac{1}{2}$. If $f\in H^s(\R^2)$, then it hold that
\[
\F_1[f](\cdot,x_2)\in L^1(\R)\ \ {\rm a.e.}\ x_2\in \R,\ \ 
\F_2[f](x_1,\cdot)\in L^1(\R)\ \ {\rm a.e.}\ x_1\in \R.
\]
\end{lemm}
\begin{proof}
We prove only for $\F_1[f]$. 
By the Cauchy-Schwarz inequality 
and Plancherel's theorem, we have
\[
\begin{split}
\left\|\|\F_1[f](\xi_1,x_2)\|_{L^1_{\xi_1}}\right\|_{L^2_{x_2}}
&\le \left\|\|\langle \xi_1\rangle^{-s}\|_{L^2_{\xi_1}}
\|\langle \xi_1\rangle^s\F_1[f](\xi_1,x_2)\|_{L^2_{\xi_1}}\right\|_{L^2_{x_2}}\\
&\lesssim \|\langle \xi_1\rangle^s\widehat{f}(\xi_1,\xi_2)\|_{L^2_{\xi}}\\
&\lesssim \|f\|_{H^s}<\infty
\end{split}
\]
for $s>\frac{1}{2}$.
Therefore, we obtain
\[
\|\F_1[f](\xi_1,x_2)\|_{L^1_{\xi_1}}<\infty\ \ {\rm a.e.}\ x_2\in \R. \qedhere
\]
\end{proof}
\begin{proof}[Proof of Proposition~\ref{ex_sc}]
We put
\[
W(x):=\int_{a_1}^{x_1}w_1(y_1,x_2)dy_1+\int_{a_2}^{x_2}w_2(a_1,y_2)dy_2
=:W_1(x)+W_2(x)
\]
for some $a_1$, $a_2\in \R$.
By $w \in L^2(\R^2)$, we have $W\in L^1_{\rm loc}(\R^2)$. 
%For $K=[p_1,q_1]\times [p_2,q_2]\subset \R^2$, 
%by the Cauchy-Schwarz inequality, we have
%\[
%\begin{split}
%\int_K|W_1(x)|dx
%&\le \int_{p_1}^{q_1}\left(\int_{p_2}^{q_2}\int_{a_1}^{x_1}|w_1(y_1,x_2)|dy_1dx_2\right)dx_1\\
%&\le \int_{p_1}^{q_1}(q_2-p_2)^{\frac{1}{2}}(x_1-a_1)^{\frac{1}{2}}\|w_1\|_{L^2}dx_1\\
%&\lesssim \|w_1\|_{L^2}<\infty
%\end{split}
%\]
%since $w_1\in L^2(\R^2)$. 
%While, by the Cauchy-Schwarz inequality 
%and the Sobolev inequality, we have
%\[
%\begin{split}
%\int_K|W_2(x)|dx
%&\le \int_K\left(\int_{a_2}^{x_2}|w_2(a_1,y_2)|dy_2\right)dx\\
%&\le \int_K(x_2-a_2)^{\frac{1}{2}}\|w_2(a_1,y_2)\|_{L^2_{y_2}}dx\\
%&\lesssim \|w_2\|_{H^s}<\infty
%\end{split}
%\]
%for $s>\frac{1}{2}$ since $w_2\in H^s(\R^2)$. 
Hence, it remains to show that $\nabla W=w$. 
Since 
\[
\partial_1W_1(x)=w_1(x),\ \ \partial_1W_2(x)=0,\ \ \partial_2W_2(x)=w_2(a_1,x_2)
\]
hold for almost all $x=(x_1,x_2)\in \R^2$, it suffices to show
\begin{equation}\label{deriv}
\partial_2W_1(x)=w_2(x)-w_2(a_1,x_2)\ \ {\rm a.e.}\ x=(x_1,x_2)\in \R^2. 
\end{equation}
Let $h\in \R$. 
Since $\F_1[w_1](\cdot ,x_2)\in L^1(\R)$ 
a.e. $x_2\in \R$ by Lemma~\ref{l1_lem}, we have
\[
\begin{split}
&\frac{W_1(x_1,x_2+h)-W_1(x_1,x_2)}{h}\\
&=\frac{1}{h}\int_{a_1}^{x_1}\left(w_1(y_1,x_2+h)-w_1(y_1,x_2)\right)dy_1\\
&=\frac{1}{h}\int_{a_1}^{x_1}
\left(\int_{\R}
\left(\F_1[w_1](\xi_1,x_2+h)-\F_1[w_1](\xi_1,x_2)\right)e^{i\xi_1y_1}
d\xi_1\right)
dy_1\\
&=\frac{1}{h}
\int_{\R}
\left(\F_1[w_1](\xi_1,x_2+h)-\F_1[w_1](\xi_1,x_2)\right)
\left(\int_{a_1}^{x_1}e^{i\xi_1y_1}dy_1\right)
d\xi_1\\
&=\frac{1}{h}\int_{\R}\left(\int_{\R}\widehat{w_1}(\xi_1,\xi_2)e^{i\xi_2x_2}
(e^{i\xi_2h}-1)d\xi_2\right)\frac{e^{i\xi_1x_1}-e^{i\xi_1a_1}}{i\xi_1}d\xi_1
=:I_h
\end{split}
\]
by Fubini's theorem. 
Furthermore, by using $\xi_2\widehat{w_1}=\xi_1\widehat{w_2}$ 
and $\F_1^{-1}\F_2^{-1}=\F_2^{-1}\F_1^{-1}$, we have
\[
\begin{split}
I_h&=
\int_{\R}\left(\int_{\R}\widehat{w_2}(\xi_1,\xi_2)\frac{e^{i\xi_2h}-1}{i\xi_2h}e^{i\xi_2x_2}
d\xi_2\right)(e^{i\xi_1x_1}-e^{i\xi_1 a_1})d\xi_1\\
&=\F_1^{-1}\F_2^{-1}\left[\widehat{w_2}(\xi_1,\xi_2)\frac{e^{i\xi_2h}-1}{i\xi_2h}\right](x_1,x_2)-\F_1^{-1}\F_2^{-1}\left[\widehat{w_2}(\xi_1,\xi_2)\frac{e^{i\xi_2h}-1}{i\xi_2h}\right](a_1,x_2)\\
&=\F_2^{-1}\F_1^{-1}\left[\widehat{w_2}(\xi_1,\xi_2)\frac{e^{i\xi_2h}-1}{i\xi_2h}\right](x_1,x_2)-\F_2^{-1}\F_1^{-1}\left[\widehat{w_2}(\xi_1,\xi_2)\frac{e^{i\xi_2h}-1}{i\xi_2h}\right](a_1,x_2)\\
&=\int_{\R}\left(\F_2[w_2](x_1,\xi_2)-\F_2[w_2](a_1,\xi_2)\right)\frac{e^{i\xi_2h}-1}{i\xi_2h}e^{i\xi_2x_2}d\xi_2. 
\end{split}
\]
Since
$\F_2[w_2](x_1,\cdot )\in L^1(\R)$ 
a.e. $x_1\in \R$ by Lemma~\ref{l1_lem}, 
we have
\[
\begin{split}
\lim_{h\rightarrow 0}I_h&=
\int_{\R}\left(\F_2[w_2](x_1,\xi_2)-\F_2[w_2](a_1,\xi_2)\right)e^{i\xi_2x_2}d\xi_2\\
&=w_2(x_1,x_2)-w_2(a_1,x_2)
\end{split}
\]
by Lebesgue's dominant convergence theorem. 
Therefore, we obtain (\ref{deriv}). 
\end{proof}
\begin{rem}
In the proof of Proposition~\ref{ex_sc}, we also used
\[
\left|\frac{e^{i\xi_2h}-1}{i\xi_2h}\right|
\le \sup_{z\in \R}\left(\left|\frac{\cos z-1}{z}\right|+\left|\frac{\sin z}{z}\right|\right)
<\infty. 
\]
This implies
\[
\widehat{w_2}(\xi_1,\xi_2)\frac{e^{i\xi_2h}-1}{i\xi_2h}\in L^2_{\xi}(\R^2)
\]
and
\[
\F_2[w_2](x_1,\xi_2)\frac{e^{i\xi_2h}-1}{i\xi_2h}\in L^1_{\xi_2}(\R)\ \ {\rm a.e.}\ 
x_1\in \R. 
\]
\end{rem}
\begin{rem}
If $w=(w_1,w_2)\in (H^s(\R^2))^2$ for $s>\frac{1}{2}$ satisfies
\[
x_2w_1(x)-x_1w_2(x)=0,\ \ {\rm a.e.}\ x\in \R^2
\]
additionally in Proposition~\ref{ex_sc}, 
then $W\in L^1_{\rm loc}(\R^2)$ given in the proof of Proposition~\ref{ex_sc} 
is radial.
Indeed, this condition with $\nabla W(x)=w(x)$ yields \eqref{sc_po_rad}.
\end{rem}
\begin{rem}
For $s\le \frac{1}{2}$, we do not know whether 
there exists a scalar potential of $w\in (H^s(\R^2))^2$ or not. 
But we point out that if $s<\frac{1}{2}$, then 
the 1D delta function appears 
in $\partial_2w_1-\partial_1w_2$ 
for some $w\in (H^s(\R^2))^2$. 
Then, the irrotational condition does not make sense 
for pointwise. 
\end{rem}
Next, we prove that $\mathcal{A}^s(\R^2)$ is a Banach space.
\begin{prop}
For $s\ge 0$,  
$\mathcal{A}^s(\R^2)$ is a closed subspace of $(H^s(\R^2))^2$. 
\end{prop}
\begin{proof}
Let $f^{(n)}=(f_1^{(n)}, f_2^{(n)})\in \mathcal{A}^s(\R^2)$ $(n=1,2,3,\cdots )$ and 
$f=(f_1,f_2)\in (H^s(\R^2))^2$. 
Assume that 
$f^{(n)}$ convergences to $f$ in $(H^s(\R^2))^2$ as $n\rightarrow \infty$. 
We prove $f\in A^s(\R^2)$, namely $f$ satisfies (\ref{rot_cond}). 
By the triangle inequality, we have
\[
\begin{split}
&\left\|\frac{x_2}{\langle x \rangle}f_1-\frac{x_1}{\langle x \rangle}f_2\right\|_{L^2}\\
&\le 
\left\|\frac{x_2}{\langle x \rangle}f_1-\frac{x_2}{\langle x \rangle}f_1^{(n)}\right\|_{L^2}
+
\left\|\frac{x_2}{\langle x \rangle}f_1^{(n)}-\frac{x_1}{\langle x \rangle}f_2^{(n)}\right\|_{L^2}
+
\left\|\frac{x_1}{\langle x \rangle}f_2^{(n)}-\frac{x_1}{\langle x \rangle}f_2\right\|_{L^2}\\
&\le \|f_1-f_1^{(n)}\|_{L^2}+\|x_2f_1^{(n)}-x_1f_2^{(n)}\|_{L^2}+\|f_2^{(n)}-f_2\|_{L^2}.
\end{split}
\]
Since $f^{(n)}$ satisfies (\ref{rot_cond}) and $f^{(n)}\rightarrow f$ in $(L^2(\R^2))^2$ 
as $n\rightarrow \infty$, 
we obtain
\[
\|x_2f_1^{(n)}-x_1f_2^{(n)}\|_{L^2}=0,\ \ 
\|f_1-f_1^{(n)}\|_{L^2}+\|f_2^{(n)}-f_2\|_{L^2}\rightarrow 0\ (n\rightarrow \infty). 
\]
Therefore, we get
\[
\left\|\frac{x_2}{\langle x \rangle}f_1-\frac{x_1}{\langle x \rangle}f_2\right\|_{L^2}=0.
\]
It implies $x_2f_1(x)-x_1f_2(x)=0$ a.e. $x\in \R^2$. 
Similarly, we obtain $\xi_2\widehat{f_1}(\xi)-\xi_1\widehat{f_2}(\xi)=0$ a.e. $\xi\in \R^2$. 
\end{proof}
\begin{proof}[\rm{\bf{Proof of Theorem~\ref{wellposed_2}.}}]
Let $(u_0,v_0,w_0)\in B_r((H_{\rm rad}^s(\R^2))^2\times (H_{\rm rad}^s(\R^2))^2 \times \mathcal{A}^s(\R^2))$ be given. 
We first prove the existence of solution to (\ref{NLS_sys}). 
Since $w_0$ satisfies (\ref{rot_cond}), by Proposition~\ref{ex_sc}, 
there exists $[W_0]\in \widetilde{H}^{s+1}_{\rm rad}$ such that 
$\nabla W_0=w_0$. 
From Theorem~\ref{wellposed_1}, 
there exists $T>0$ and a solution $(u,v,[W])\in \mathcal{X}^{s}_T$ 
to (\ref{NLS_rad}) with $(u,v,[W])|_{t=0}=(u_0,v_0,[W_0])$. 
Since
\[
\|[W_0]\|_{\widetilde{H}^{s+1}}=\|w_0\|_{H^s}\le r, 
\]
the existence time $T$ is decided by $r$. 
We put $w=\nabla W$. 
Then, $w\in X^{s,b}_{\gamma, T}$ satisfying 
\[
\|w\|_{X^{s,b}_{\gamma, T}}
=\|[W]\|_{\widetilde{X}^{s+1,b}_{\gamma, T}}
\le R, 
\]
where $R$ is as in the proof of Theorem~\ref{wellposed_1},
and 
$(u,v,w)$ satisfies (\ref{NLS_sys}) since $\Delta W=\nabla \cdot w$. 
Furthermore, we have
\[
\partial_1w_2-\partial_2w_1=\partial_1(\partial_2W)-\partial_2(\partial_1W)=0
\]
and
\[
x_1w_2-x_2w_1=(x_1\partial_2-x_2\partial_1)W=0
\]
because $W$ is radial with respect to $x$. 
Therefore, $w(t)\in \mathcal{A}^s(\R^2)$ for any $t\in [0,T]$. 

Next, we prove the uniqueness of the solution in $B_R(\mathcal{Y}^{s,b}_T)$, 
where
\[
\begin{split}
\mathcal{Y}^{s,b}_T&:=(X^{s,b}_{\alpha,{\rm rad},T})^2\times (X^{s,b}_{\beta,{\rm rad},T})^2\times Y^{s,b}_{\gamma,T},\\
Y^{s,b}_{\gamma, T}&:=
\{w=(w_1,w_2)\in (X^{s,b}_{\gamma,T})^2|w(t)\ {\rm satisfies}\ (\ref{rot_cond})\ 
{\rm for\ any\ }t\in[0,T]\}. 
\end{split}
\] 
Let $(u^{(1)},v^{(1)},w^{(1)})$, $(u^{(2)},v^{(2)},w^{(2)})\in B_R(\mathcal{Y}^{s,b}_T)$ are solution to (\ref{NLS_sys}) 
with initial data $(u_0,v_0,w_0)$. 
Then by Proposition~\ref{ex_sc},  there exists $[W^{(1)}]$, $[W^{(2)}]\in \widetilde{X}^{s+1,b}_{\gamma,{\rm rad},T}$ 
such that $w^{(1)}=\nabla W^{(1)}$, $w^{(2)}=\nabla W^{(2)}$. 
By substituting $w^{(j)}=\nabla W^{(j)}$ 
in both sides of the integral form of (\ref{NLS_sys}), 
$(u^{(j)},v^{(j)},W^{(j)})$ $(j=1,2)$ satisfy
\[
\begin{split}
u^{(j)}(t)&=e^{it\alpha \Delta}u_0+i\int_0^te^{i(t-t')\alpha \Delta}(\Delta W^{(j)}(t'))u^{(j)}(t')dt'\ \ {\rm in}\ (H^s(\R^2))^2,\\
v^{(j)}(t)&=e^{it\beta \Delta}v_0+i\int_0^te^{i(t-t')\beta \Delta}(\Delta \overline{W^{(j)}(t')})v^{(j)}(t')dt'\ \ {\rm in}\ (H^s(\R^2))^2,\\
\nabla W^{(j)}(t)&=e^{it\gamma \Delta}w_0-i\int_0^te^{i(t-t')\gamma \Delta}\nabla (u^{(j)}(t')\cdot \overline{v^{(j)}(t')})dt'\ \ {\rm in}\ H^s(\R^2). 
\end{split}
\]
Therefore, by the same argument as in the proof of Theorem~\ref{wellposed_1}, 
we have
\[
\begin{split}
\|u^{(1)}-u^{(2)}\|_{X^{s,b}_{\alpha ,T}}
&\le \frac{1}{4}(\|w^{(1)}-w^{(2)}\|_{X^{s,b}_{\gamma,T}}
+\|v^{(1)}-v^{(2)}\|_{X^{s,b}_{\beta,T}})\\
\|v^{(1)}-v^{(2)}\|_{X^{s,b}_{\beta ,T}}
&\le \frac{1}{4}(\|w^{(1)}-w^{(2)}\|_{X^{s,b}_{\gamma,T}}
+\|u^{(1)}-u^{(2)}\|_{X^{s,b}_{\alpha,T}})\\
\|w^{(1)}-w^{(2)}\|_{X^{s,b}_{\gamma ,T}}
&\le \frac{1}{4}(\|u^{(1)}-u^{(2)}\|_{X^{s,b}_{\alpha,T}}
+\|v^{(1)}-v^{(2)}\|_{X^{s,b}_{\beta,T}})
\end{split}
\]
since $w^{(1)}-w^{(2)}=\nabla (W^{(1)}-W^{(2)})$. 
This implies $(u^{(1)},v^{(1)},w^{(1)})=(u^{(2)},v^{(2)},w^{(2)})$ 
on $[0,T]$. 
The continuous dependence on initial data can be obtained 
by the similar argument.  
\end{proof}
%
%
%%%%%%%%%%%%%%%%%%%%%%%%%%%%%%%%%%%%%%%%%%%%%%%%%%%%%%%%%%%%%%%%%%%%%%%%%%%%%%%%%
%%%%%%%%%%%%%%%%%%%%%%%%%%%%%%%%%%%%%%%%%%%%%%%%%%%%%%%%%%%%%%%%%%%%%%%%%%%%%%%%%
%%%%%%%%%%%%%%%%%%%%%%%%%%%%%%%%%%%%  Section 4   %%%%%%%%%%%%%%%%%%%%%%%%%%%%%%%%%%
%%%%%%%%%%%%%%%%%%%%%%%%%%%%%%%%%%%%%%%%%%%%%%%%%%%%%%%%%%%%%%%%%%%%%%%%%%%%%%%%%
%%%%%%%%%%%%%%%%%%%%%%%%%%%%%%%%%%%%%%%%%%%%%%%%%%%%%%%%%%%%%%%%%%%%%%%%%%%%%%%%%
%
\section{The lack of the twice differentiability of the flow map}
The following proposition implies Theorem~\ref{ill-posed}.
\begin{prop}\label{notC2}
Let $d=2$ and $0<T\ll 1$. 
Assume $\theta =0$ and $s<\frac{1}{2}$. 
For every $C>0$, there exist $f$, $g\in H_{\rm rad}^{s}(\R^{2})$ such that
\begin{equation}\label{flow_map_estimate}
\sup_{0\leq t\leq T}\left\|\int_{0}^{t}e^{i(t-t')\gamma \Delta}\nabla \left((e^{it'\alpha \Delta}f)(\overline{e^{it'\beta \Delta}g})\right)dt'\right\|_{H^{s}}
\geq C\|f\|_{H^{s}}\|g\|_{H^{s}}. 
\end{equation}
\end{prop}
\begin{proof}
Let $N\gg 1$ and $p:=\frac{\gamma}{\alpha -\gamma}$\ $(\ne 0)$. 
We note that $p$ is well-defined since $\theta =0$ implies 
$\kappa \ne 0$ for $\alpha$, $\beta$, $\gamma \in \R\backslash\{0\}$. 
For simplicity, we assume $p>0$. 
Put
\[
\begin{split}
D_1&:=\{\xi \in \R^2|\ N\le |\xi |\le N+1\},\ \ 
D_2:=\{\xi \in \R^2|\ p^{-1}N\le |\xi|\le p^{-1}N+1\},\\
D&:=\{\xi \in \R^2|\ (1+p^{-1})N+1\le |\xi|\le (1+p^{-1})N+1+2^{-10}\}.
\end{split}
\]
We define the functions $f$ and $g$ as
\[
\widehat{f}(\xi ):=N^{-s-\frac{1}{2}}\ee_{D_1}(\xi),\ \ \widehat{g}(\xi ):=N^{-s-\frac{1}{2}}\ee_{D_2}(\xi).
\]
Clearly, we have $\|f\|_{H^s}\sim \|g\|_{H^s}\sim 1$ and $f$, $g$ are radial. 
For $\xi =(\xi_1,\xi_2)\in \R^2$ and $\eta=(\eta_1,\eta_2)\in \R^2$, we define
\[
\begin{split}
\Phi (\xi,\eta)&:= \alpha |\eta|^2-\beta |\xi-\eta|^2-\gamma |\xi|^2\\
&=(\alpha -\gamma)|\eta-p(\xi-\eta)|^2\\
&=(\alpha-\gamma)\left\{\left(\eta_1-p(\xi_1-\eta_1)\right)^2+\left(\eta_2-p(\xi_2-\eta_2)\right)^2\right\}
\end{split}
\]
because $\theta =0$ implies $\frac{\beta +\gamma}{\alpha -\gamma}=-\left(\frac{\gamma}{\alpha -\gamma}\right)^2$. 
We will show 
\[
\sup_{0\le t\le T}\left\|\int_{0}^{t}e^{i(t-t')\gamma \Delta}\nabla \left((e^{it'\alpha \Delta}f)(\overline{e^{it'\beta \Delta}g})\right)dt'\right\|_{H^{s}}
\gtrsim N^{-s+\frac{1}{2}}. 
\]
We calculate that
\[
\begin{split}
&\left\|\int_{0}^{t}e^{i(t-t')\gamma \Delta}\nabla \left((e^{it'\alpha \Delta}f)(\overline{e^{it'\beta \Delta}g})\right)dt'\right\|_{H^{s}}\\
&\gtrsim N^{-s}\left\|\ee_D(\xi)\int_0^t\int_{\R^2}e^{-it'\Phi(\xi,\eta)}\ee_{D_1}(\eta)\ee_{D_2}(\xi-\eta)d\eta\right\|_{L^2_{\xi}}\\
&\ge N^{-s}\left\|\ee_D(\xi)\int_0^t\int_{\R^2}\cos (t'\Phi(\xi,\eta))\ee_{D_1}(\eta)\ee_{D_2}(\xi-\eta)d\eta\right\|_{L^2_{\xi}}\\
&=:N^{-s}\left\|F(\xi )\right\|_{L^2_{\xi}}.
\end{split}
\]
Let $R:\R^2\rightarrow \R^2$ be a rotation operator. 
Since $\Phi (\xi,\eta)=\Phi(R\xi,R\eta)$ and $\ee_D$, $\ee_{D_1}$, $\ee_{D_2}$ are radial, 
we can see
\[
\begin{split}
F(\xi)&=
\ee_D(\xi)\int_{\R^2}\frac{\sin (t\Phi(\xi,\eta))}{\Phi (\xi,\eta)}\ee_{D_1}(\eta)\ee_{D_2}(\xi-\eta)d\eta\\
&=\ee_D(R\xi)\int_{\R^2}\frac{\sin (t\Phi(R\xi,R\eta))}{\Phi (R\xi,R\eta)}\ee_{D_1}(R\eta)\ee_{D_2}(R\xi-R\eta)d\eta\\
&=\ee_D(R\xi)\int_{\R^2}\frac{\sin (t\Phi(R\xi,\eta))}{\Phi (R\xi,\eta)}\ee_{D_1}(\eta)\ee_{D_2}(R\xi-\eta)d\eta\\
&=F(R\xi).
\end{split}
\]
It implies that $F$ is radial. 
Therefore, there exists $G:\R\rightarrow \R$ such that $F(\xi)=G(|\xi|)$. 
We note that
\[
\|F(\xi )\|_{L^2_{\xi}}=\|G(r)r^{\frac{1}{2}}\|_{L^2((0,\infty))}\gtrsim N^{\frac{1}{2}}\inf_{r>0}|G(r)|=N^{\frac{1}{2}}\inf_{(\xi_1,0) \in D}|F(\xi_1,0)|
\]
since ${\rm supp}G\subset [(1+p^{-1})N+1, (1+p^{-1})N+1+2^{-10}]$. 
Hence, it suffices to show that 
\begin{equation}\label{notC2_desire}
|F(\xi_c)|\gtrsim t^{\frac{1}{2}}
\end{equation}
for any $c\in [0,2^{-10}]$ and $0<t\ll 1$, where $\xi_c:=((1+p^{-1})N+1+c,0)\in \R^2$. 
Simple calculation gives
\begin{equation}\label{notC2_modu}
\Phi (\xi_c,\eta)=(\alpha-\gamma)\left\{\left((1+p)(\eta_1-N)-p(1+c)\right)^2+(1+p)^2\eta_2^2\right\}.
\end{equation}
We also observe that
\[
\begin{split}
&\ee_{D_1}(\eta)\ee_{D_2}(\xi_c-\eta)\ne 0\\
\Longrightarrow\ &\eta_1\le N+1\ {\rm and}\ (1+p^{-1})N+1+c-\eta_1\le p^{-1}N+1\\
\Longrightarrow\ &N+c\le \eta_1\le N+1.
\end{split}
\]
Let $\epsilon >0$ be small. We define a new set $E$ as
\[
E:=D_1\cap \{\eta =(\eta_1,\eta_2)\in \R^2|\ N+c\le \eta_1\le N+1\},
\]
and we decompose $E$ into four sets:
\[
\begin{split}
E_1&=\left\{
\begin{split}
&N+c+p^{-1}N+1-\sqrt{(p^{-1}N+1)^2-N^{2\epsilon}}\le \eta_1 <\sqrt{(N+1)^2-N^{2\epsilon}},\\
&|\eta_2|\le N^{\epsilon}
\end{split}
\right\},\\
E_2&=\{N+c\le \eta_1<N+c+p^{-1}N+1-\sqrt{(p^{-1}N+1)^2-N^{2\epsilon}},\ |\eta_2|\le N^{\epsilon}\}\cap E,\\
E_3&=\{\sqrt{(N+1)^2-N^{2\epsilon}}\le \eta_1\le N+1,\ |\eta_2|\le N^{\epsilon}\}\cap E,\\
E_4&=\{N^{\epsilon}<|\eta_2|\}\cap E.
\end{split}
\]
We can easily show that $E_i\cap E_j=\emptyset$ if $i\ne j$. 
Furthermore, we can obtain $E_1\subset E$ and
\[
\ee_{D_1}(\eta)\ee_{D_2}(\xi_c-\eta)=1
\]
for any $\eta \in E_1$. 
We observe that
\[
\begin{split}
|F(\xi_c)|
&\ge \left|\int_{\R^2}\frac{\sin (t\Phi (\xi_c,\eta))}{\Phi (\xi_c,\eta)}\ee_{E_1}(\eta)d\eta\right|
-\sum_{j=2}^4\int_{\R^2}\left|\frac{\sin (t\Phi (\xi_c,\eta))}{\Phi (\xi_c,\eta)}\right|\ee_{E_j}(\eta )d\eta\\
&=:I_1-\sum_{j=2}^4I_j.
\end{split}
\]

We first consider $I_1$. Let
\[
c':=p^{-1}N+1-\sqrt{(p^{-1}N+1)^2-N^{2\epsilon}},\ \ c'':=N+1-\sqrt{(N+1)^2-N^{2\epsilon}}.
\]
Obviously, it holds $c'\sim c''\sim N^{-1+2\epsilon}$. We calculate that
%\oka{
\[
\begin{split}
%&\int_{\R^2}\frac{\sin (t\Phi (\xi_c,\eta))}{\Phi(\xi_c,\eta)} \oka{\ee_{E_1}(\eta)} d\eta \\
I_1
&=2 \left| \int_{N+c'+c''}^{N+1-c''}\left(\int_0^{N^{\epsilon}}\frac{\sin (t\Phi (\xi_c,\eta))}{\Phi(\xi_c,\eta)}d\eta_2\right)d\eta_1 \right| \\
&
=\frac{2}{(1+p) |\alpha-\gamma|} \left| \int_{N+c'+c''}^{N+1-c''}\left(\int_0^{(1+p)N^{\epsilon}}\frac{\sin (\tau (q(\eta_1)+\eta_2^2))}{q(\eta_1)+\eta_2^2}d\eta_2\right)d\eta_1 \right|,
\end{split}
\]
where $\tau:=|\alpha -\gamma|t$ and $q(\eta_1):=\left((1+p)(\eta_1-N)-p(1+c)\right)^2$.%
%}
Therefore, if we obtain
\begin{equation}\label{notC2_I_1_est}
\inf_{\eta_1\in [N+c'+c'',N+1-c'']} \int_0^{(1+p)N^{\epsilon}}\frac{\sin (\tau (q(\eta_1)+\eta_2^2))}{q(\eta_1)+\eta_2^2}d\eta_2
\gtrsim t^{\frac{1}{2}}, 
\end{equation}
then we get $I_1\gtrsim t^{\frac{1}{2}}$. 
Let $t>0$ be small. We fix $\eta_1\in [N+c'+c'',N+1-c'']$ and 
write $q(\eta_1)=q$ for simplicity.
Clearly, we have $0\le q\lesssim 1$. 
We easily verify that if we restrict $\eta_2$ as $0\le \eta_2\le \sqrt{\pi \tau^{-1}-q}$, then 
we have $\sin (\tau (q+\eta_2^2))\ge 0$
and $\frac{\sin (\tau (q+\eta_2^2))}{q+\eta_2^2}$ is monotone decreasing. 
Similarly, if $\sqrt{\pi \tau^{-1}-q}\le \eta_2\le \sqrt{2\pi \tau^{-1}-q}$, then we see $\sin (\tau (q+\eta_2^2))\le 0$. 
We calculate
\[
\begin{split}
\int_0^{\sqrt{2\pi\tau^{-1}-q}}\frac{\sin (\tau (q+\eta_2^2))}{q+\eta_2^2}d\eta_2
&\ge \int_0^{\sqrt{\pi\tau^{-1}-q}}\frac{\sin (\tau (q+\eta_2^2))}{q+\eta_2^2}d\eta_2
-\int_{\sqrt{\pi\tau^{-1}-q}}^{\sqrt{2\pi\tau^{-1}-q}}\frac{1}{q+\eta_2^2}d\eta_2\\
&\ge \frac{2\tau}{\pi}\int_0^{\sqrt{\pi (2\tau)^{-1}-q}}d\eta_2-\frac{\tau}{\pi}\int_{\sqrt{\pi \tau^{-1}-q}}^{\sqrt{2\pi \tau^{-1}-q}}d\eta_2\\
&=\frac{\tau}{\pi}\left(2\sqrt{\pi (2\tau)^{-1}-q}-\sqrt{2\pi \tau^{-1}-q}+\sqrt{\pi \tau^{-1}-q}\right)\\
&\gtrsim t^{\frac{1}{2}}. 
\end{split}
\]
The last estimate is verified by the smallness of $\tau =|\alpha -\gamma |t$. 
We also see
\[
\int_{\sqrt{2n\pi \tau^{-1}-q}}^{\sqrt{2(n+1)\pi \tau^{-1}-q}}\frac{\sin (\tau (q+\eta_2^2))}{q+\eta_2^2}d\eta_2
\gtrsim \frac{t^{\frac{1}{2}}}{n^2}
\]
for any $n\in \N$. Therefore, we obtain (\ref{notC2_I_1_est}). 

Next, we consider $I_2$, $I_3$, and $I_4$. 
Since $|E_2|$, $|E_3|\lesssim N^{-1+3\epsilon}$, we easily observe that
\[
I_2+I_3\lesssim tN^{-1+3\epsilon}. 
\]
For $I_4$, we observe that
\[
I_4=\int_{E_4}\left|\frac{\sin (t\Phi (\xi_c,\eta))}{\Phi (\xi_c,\eta)}\right|d\eta
\lesssim t\int_{N+c}^{N+1}\left(\int_{N^{\epsilon}}^{\infty}\frac{1}{\eta_2^2}d\eta_2\right)d\eta_1\lesssim tN^{-\epsilon}. 
\]
By the above argument, we obtain
\[
|F(\xi_c)|\ge I_1-\sum_{j=2}^4I_j\gtrsim t^{\frac{1}{2}}-t(N^{-1+3\epsilon}+N^{-\epsilon})\gtrsim t^{\frac{1}{2}}
\]
for small $0<t\ll 1$. 
\end{proof}
\section*{acknowledgements}
This work was financially supported by 
JSPS KAKENHI Grant Number JP16K17624 and JP17K14220, 
Program to Disseminate Tenure Tracking System from the Ministry of Education, Culture, Sports, Science and Technology, 
and the DFG through the CRC 1283 "Taming uncertainty and
profiting from randomness and low regularity in analysis, stochastics
and their
applications".  

\end{document}